\documentclass[11pt]{amsart}
\usepackage[leqno]{amsmath}
\usepackage{amsmath,amsfonts,amsthm,amssymb}
\usepackage{hyperref}
\usepackage[alphabetic]{amsrefs}
\usepackage{epsfig}
\usepackage{color}
\newtheorem{lemma}{Lemma}[section]
\newtheorem{proposition}[lemma]{Proposition}
\newtheorem{theorem}[lemma]{Theorem}
\newtheorem{corollary}[lemma]{Corollary}

\theoremstyle{definition}

\newtheorem{remark}[lemma]{Remark}

\newcommand{\commentout}[1]{}
\usepackage{graphicx}

\newcommand {\mr}{\mathrm}
\newcommand {\lk}{\left\{}
\newcommand {\rk}{\right\}}

\newcommand {\xms}{\vspace{0.1cm}}
\newcommand {\eg}{\underline EG}
\newcommand {\cx}{\overline{\bf X}}
\newcommand {\ov}{\overline}
\newcommand{\cal}{\mathcal}

\begin{document}

\thispagestyle{empty}

\centerline{\Large\bf Dismantlability of weakly systolic complexes and applications}

\vspace{6mm}

\centerline{{\sc Victor Chepoi$^{\small 1}$}  and {\sc Damian Osajda}$^{\small 2}$}

\vspace{3mm}
\medskip
\centerline{$^{1}$Laboratoire d'Informatique Fondamentale,}
\centerline{Aix-Marseille Universit\'e,}
\centerline{Facult\'e des Sciences de Luminy,} \centerline{F-13288
Marseille Cedex 9, France} \centerline{chepoi@lif.univ-mrs.fr}

\medskip
\centerline{$^2$ Instytut Matematyczny, Uniwersytet Wroc{\l}awski}
\centerline{pl.\ Grunwaldzki 2/4, 50-384 Wroc{\l}aw, Poland}
\centerline{and}
\centerline{Universit\"at Wien, Fakult\"at f\"ur Mathematik}
\centerline{Nordbergstra\ss e 15, 1090 Wien, Austria}
\centerline{dosaj@math.uni.wroc.pl}

\bigskip

\begin{footnotesize} \noindent {\bf Abstract.}
The main goal of this paper is proving the fixed point theorem for finite groups acting on weakly systolic complexes. As corollaries we obtain results concerning classifying spaces for the family of finite subgroups of weakly systolic groups and conjugacy classes of finite subgroups. As immediate consequences we get new results on systolic complexes and groups. \\
The fixed point theorem is proved by using a graph-theoretical tool---dismantlability. In particular we show that $1$--skeleta of weakly systolic complexes, i.e.\ weakly bridged graphs, are dismantlable. On the way we show numerous characterizations of weakly bridged graphs and weakly systolic complexes.
\end{footnotesize}


\bigskip


\section{Introduction}
\label{intro}

In his seminal paper \cite{G}, among many other results, Gromov gave a pretty combinatorial characterization of CAT(0) cubical
complexes as simply connected cubical complexes in which the links of vertices are simplicial flag complexes. Based on this result,
\cite{Ch_CAT,Rol} established a bijection between the $1$--skeletons of CAT(0) cubical complexes and the median graphs, well-known
in metric graph theory \cite{BaCh_survey}. A similar combinatorial characterization of CAT(0) simplicial complexes having regular Euclidean
simplices as cells seems to be out of reach. Nevertheless, Chepoi \cite{Ch_CAT} characterized the bridged complexes (i.e., the simplicial complexes having
bridged graphs as $1$--skeletons) as the simply connected simplicial complexes in which the links of vertices are flag complexes without embedded $4$-- and
$5$--cycles; the bridged graphs are exactly the graphs which satisfy one of the basic features of CAT(0) spaces: the balls around
convex sets are convex. Bridged graphs have been introduced in \cite{FaJa,SoCh} as graphs without embedded isometric cycles of
length greater than $3$ and have been further investigated in several graph-theoretical and algebraic papers; cf.\
\cite{AnFa,BaCh_weak,Ch_bridged,Po,Po1} and the survey \cite{BaCh_survey}. Januszkiewicz-\'Swi\c atkowski \cite{JanSwi} and Haglund \cite{Ha}
rediscovered this class of simplicial complexes (they call them {\it systolic complexes}) and used them (and groups acting on them geometrically---{\it systolic groups}) fruitfully in the context of geometric group theory. Systolic complexes and groups turned out to be good combinatorial
analogs of CAT(0) (nonpositively curved)  metric spaces and groups; cf.\
\cite{Ha,JanSwi,O-ciscg,OsPr,Pr2,Pr3}.

One of the characteristic features of systolic complexes, related to the convexity of balls around convex sets, is the following  $SD_n(\sigma^*)$
property introduced in \cite{Osajda}: {\it  if a simplex $\sigma$ of a simplicial complex $\bf X$ is located in the sphere of radius $n+1$ centered
at some simplex $\sigma^*$ of $\bf X$, then the set of all vertices $x$ such that $\sigma\cup\{ x\}$ is a simplex and $x$ has distance $n$ to $\sigma^*$ is a
nonempty simplex $\sigma_0$ of $\bf X$.} Relaxing this condition, Osajda \cite{Osajda} called a simplicial complex $\bf X$ {\it weakly systolic} if the property
$SD_n(\sigma^*)$ holds whenever $\sigma^*$ is a vertex (i.e., a $0$--dimensional simplex) of $\bf X$. He further showed that this $SD_n$ property is
equivalent with the $SD_n(\sigma^*)$ property in which $\sigma^*$ is a vertex and $\sigma$ is a vertex or an edge (i.e., an $1$--dimensional simplex)
of $\bf X$.
Finally it is showed in \cite{Osajda} that weakly systolic
complexes can be characterized as simply connected simplicial complexes satisfying some local combinatorial conditions, cf.\ also Theorem A below. This is analogous to the cases of $CAT(0)$ cubical complexes and systolic complexes.
In graph-theoretical terms, the $1$--skeletons of weakly systolic complexes (which we call {\it weakly bridged graphs})
satisfy the so-called triangle and quadrangle conditions \cite{BaCh_weak}, i.e., like median and bridged graphs, the weakly bridged graphs are weakly modular graphs. From the results of \cite{Osajda} and of the present paper it follows that the properties of weakly systolic complexes resemble very much the properties of spaces of non-positive curvature.

The initial motivation of \cite{Osajda} for introducing weakly systolic complexes  was to exibit a class of  simplicial complexes with some kind of
simplicial nonpositive curvature that will include the systolic complexes  and some other classes of complexes appearing in the context of geometric group theory. As we noticed already, systolic complexes are weakly systolic. Moreover, for every simply connected locally $5$--large cubical complex (i.e.\ $CAT(-1)$ cubical complex \cite{G}) there exists a canonically associated simplicial complex, which is weakly systolic \cite{Osajda}. In particular, the class of {\it weakly systolic groups}, i.e., groups acting geometrically by automorphisms on weakly systolic complexes, contains the class of $CAT(-1)$ cubical groups and is therefore essentially bigger than the class of systolic groups; cf.\ \cite{O-ciscg}. Other classes of weakly systolic groups are presented in \cite{Osajda}. The ideas and results from \cite{Osajda} permit the construction in \cite{O2} of new examples of Gromov hyperbolic groups of
arbitrarily large (virtual) cohomological dimension.   Furthermore, Osajda \cite{Osajda} and Osajda-\' Swi\c atkowski \cite{OS} provide new examples of high dimensional groups with interesting asphericity properties. On the other hand, as we will show below, the class of weakly systolic complexes seems also to appear naturally in the context of graph theory and has not been studied before from this point of view.

In this paper, we present further characterizations and properties of
weakly systolic complexes and their $1$--skeletons, weakly bridged graphs.
Relying on techniques from graph theory we establish dismantlability of
locally-finite weakly bridged graphs. This result is used to show  some
interesting nonpositive-curvature-like properties of weakly systolic complexes
and groups (see \cite{Osajda} for other properties of this kind). As
corollaries, we also get new results about systolic complexes and
groups. We conclude this introductory section with the formulation of our main
results (see respective sections for all missing definitions and notations as well as
for other related results).

We start with a characterization of weakly systolic complexes
proved in Section \ref{char}:

\medskip
\noindent
{\bf Theorem A.}  {\it For a flag simplicial complex $\bold X$ the following conditions are equivalent:
\begin{itemize}
\item[(a)] ${\bold X}$ is weakly systolic;
\item[(b)] the $1$--skeleton of ${\bold X}$ is a weakly modular graph without induced $C_4$;
\item[(c)] the $1$--skeleton of ${\bold X}$  is a weakly modular graph with convex balls;
\item[(d)] the $1$--skeleton of ${\bold X}$ is a graph with convex balls in which any $C_5$ is included in a $5$--wheel $W_5$;
\item[(e)] $\bold X$ is simply connected,  satisfies the $\widehat{W}_5$--condition, and does not contain induced $C_4.$
\end{itemize}
}
\medskip

In Section \ref{dismantlability} we prove the following result:

\medskip\noindent
{\bf Theorem B.}  {\it Any LexBFS ordering of vertices of a locally finite weakly systolic complex $\bf X$ is a dismantling ordering of
its $1$--skeleton.}
\medskip

This  result allows us to prove in Section \ref{fixedpt} the following fixed point theorem concerning group actions:

\medskip\noindent
{\bf Theorem C.}  {\it
Let $G$ be a finite group acting by simplicial automorphisms
on a locally finite weakly systolic complex ${\bold X}$. Then there exists
a simplex $\sigma \in {\bold X}$ which is invariant under the action
of $G$.}
\medskip

The barycenter of an invariant simplex is a point fixed by $G$.
An analogous theorem holds in the case of $CAT(0)$ spaces; cf.\
\cite[Corollary 2.8]{BrHa}.
As a direct corollary of Theorem C,  we get the fixed point theorem for systolic complexes.
This was conjectured by Januszkiewicz-\' Swi\c atkowski (personal communication) and Wise \cite{Wi}, and later was formulated
in the collection of open questions \cite[Conjecture 40.1 on page 115]{Chat}.
A partial result in the systolic case was proved by Przytycki \cite{Pr2}.
In fact, in Section \ref{final}, based on a result of Polat \cite{Po} for bridged graphs,
we prove an even stronger version of the fixed point theorem in this case.

The use of dismantlability of the underlying graph to prove the fixed point theorem for finite group actions is, due to our knowledge, a novelty brought by the current paper. It should be noticed, that there are well known examples of contractible, or even collapsible simplicial complexes admitting finite group actions without fixed points. Thus it seems that dismantlability is a right strengthening of those properties in the context of fixed point results. Subsequently, many other complexes studied in connection with group actions have dismantling properties. There, this approach gives new results concerning sets of fixed points; cf.\ e.g.\ \cite{PS}.

There are several important group theoretical consequences of Theorem C. The first one follows directly
from this theorem and  \cite[Remarks 7.7$\&$7.8]{Pr2}.

\medskip\noindent
{\bf Theorem D.}  {\it Let $k\geq 6$. Free products of $k$--systolic groups amalgamated over finite subgroups
are $k$--systolic. HNN extensions of $k$--systolic groups over finite subgroups are $k$--systolic.}
\medskip

The following result (Corollary \ref{conj} below) also has its $CAT(0)$ counterpart; cf.\ \cite[Corollary 2.8]{BrHa}:

\medskip\noindent
{\bf Corollary.}   {\it Let $G$ be a weakly systolic group.
Then $G$ contains only finitely many conjugacy classes of finite subgroups.}
\medskip

The next important consequence of the fixed point theorem concerns classifying spaces for proper group actions.
Recall that if a group $G$ acts properly on a space $\bf X$ such that the fixed point
set for any finite subgroup of $G$ is contractible (and therefore non-empty), then we say that $\bf X$ is a \emph {model for $\eg$}---the classifying space
for finite groups. If additionally the action is cocompact, then $\bf X$ is
a \emph {finite model for $\eg$}. A (finite) model for $\eg$ is in a sense a
``universal'' $G$--space (see \cite{Lu} for details).
The following theorem is a direct consequence of Theorem C and
Proposition \ref{inv set contr} below.

\medskip\noindent
{\bf Theorem E.}  {\it
Let $G$ act properly by simplicial automorphisms on a finite dimensional weakly systolic complex
$\bf X$. Then $\bf X$ is a finite dimensional model for $\eg$. If, moreover, the action of $G$ on $\bf X$ is
cocompact, then $\bf X$ is a finite model for $\eg$.}
\medskip

As an immediate consequence we get an analogous result about $\eg$ for
systolic groups. This was conjectured in \cite[Chapter 40]{Chat}.
Przytycki \cite{Pr3} showed that the Rips complex (with the constant
at least $5$) of a systolic complex is an $\eg$ space. Our result gives a systolic---and thus much nicer---model
of $\eg$ in that case. In particular the new model allows us to construct other classifying spaces; cf. \cite{O3}.
\medskip

In the final Section \ref{final} we present some further results about systolic complexes and groups.
Besides a stronger version of Theorem C, we remark on another approach
to this theorem initiated by Zawi\' slak \cite{Z} and Przytycki \cite{Pr2}.
In particular, our Proposition \ref{round} proves their conjecture about round complexes; cf.\ \cite[Conjecture 3.3.1]{Z} and \cite[Remark 8.1]{Pr2}.
Finally, we show (cf.\ the end of Section \ref{final}) how our results about $\eg$ apply to the questions of existence
of  particular boundaries of systolic groups (and thus to the Novikov conjecture for systolic groups with torsion).
This relies on earlier results of Osajda-Przytycki \cite{OsPr}.

\section{Preliminaries}

\subsection{Graphs and simplicial complexes} We continue with basic definitions used in
this paper concerning graphs and simplicial complexes (see \cite{Diestel} for graph theoretical notions used in this paper).
All graphs $G=(V,E)$ occurring here
are undirected, connected, and
without loops or multiple edges.
A graph $G$ is {\it complete} if any two of its vertices are connected by an edge.
A graph $H=(V',E')$ is an {\it induced subgraph} of the graph $G$ if $V'\subseteq V$, and $uv\in E'$ iff $uv\in E$.
The {\it distance} $d(u,v)$ between
two vertices $u$ and $v$ is the length of a shortest $(u,v)$--path,
and the {\it interval} $I(u,v)$ between $u$ and $v$ consists of all
vertices on shortest $(u,v)$--paths, that is, of all vertices
(metrically) {\it between} $u$ and $v$:
$$I(u,v)=\{ x\in V: d(u,x)+d(x,v)=d(u,v)\}.$$
An induced subgraph of $G$ (or the corresponding vertex set $A$) is
called {\it convex} if it includes the interval of $G$ between any
of its vertices.  By the {\it convex hull} conv$(W)$ of $W\subseteq V$ in $G$
we mean the smallest convex subset of $V$ (or induced subgraph of
$G$) that contains $W.$ An {\it isometric subgraph} of $G$ is an
induced subgraph  in which the distances between any two vertices
are the same as in $G.$  In particular, convex subgraphs are
isometric. The {\it(open) neighborhood} $N(x)$ of a vertex $x$ consists of all vertices $y$ adjacent to $x$ in $G$. The {\it ball} (or disk) $B_r(x)$ of center
$x$ and radius $r\ge 0$ consists of all vertices of $G$ at distance
at most $r$ from $x.$ In particular,  the unit ball $B_1(x)$
comprises $x$ and the neighborhood $N(x)$ of $x.$ The {\it sphere} $S_r(x)$ of center
$x$ and radius $r\ge 0$ consists of all vertices of $G$ at distance
exactly $r$ from $x.$ The ball $B_r(S)$
centered at a convex set $S$ is the union of all balls $B_r(x)$ with
centers $x$ from $S.$ The {\it sphere} $S_r(S)$ of center
$S$ and radius $r\ge 0$ consists of all vertices of $G$ at distance
exactly $r$ from $S.$

A graph $G$ is called  {\it thin} if for any two nonadjacent vertices $u,v$ of $G$
any two neighbors of $v$ in the interval $I(u,v)$ are adjacent. A graph $G$ is {\it weakly modular}
\cite{BaCh_weak,BaCh_survey} if its distance function $d$ satisfies the
following conditions:

\medskip\noindent
{\it Triangle condition} (T):  for any three vertices $u,v,w$ with
$1=d(v,w)<d(u,v)=d(u,w)$ there exists a common neighbor $x$ of $v$
and $w$ such that $d(u,x)=d(u,v)-1.$

\medskip\noindent
{\it Quadrangle condition} (Q): for any four vertices $u,v,w,z$ with
$d(v,z)=d(w,z)=1$ and  $2=d(v,w)\le d(u,v)=d(u,w)=d(u,z)-1,$ there
exists a common neighbor $x$ of $v$ and $w$ such that
$d(u,x)=d(u,v)-1.$

An abstract {\it simplicial complex} ${\bold X}$ is a collection of
sets (called {\it simplices}) such that $\sigma\in {\bold X}$ and
$\sigma'\subseteq \sigma$ implies $\sigma'\in {\bold X}.$
The {\it geometric realization} $\vert {\bold X}\vert$ of a
simplicial complex is the  polyhedral
complex obtained by replacing every face $\sigma$  of $\bf X$ by a ``solid"
regular simplex  $|\sigma|$  of the same
dimension such that realization commutes with intersection, that is,
$|\sigma'|\cap |\sigma''|=|\sigma'\cap \sigma''|$ for any two simplices
$\sigma'$ and $\sigma''.$  Then $\vert {\bold X}\vert=\bigcup\{
|\sigma|:\sigma\in {\bold X}\}.$  $\bold X$ is called
{\it simply connected} if it is connected and if every continuous
mapping of the $1$--dimensional sphere $S^1$ into $|{\bold X}|$ can
be extended to a continuous mapping of the disk $D^2$ with boundary
$S^1$ into $|{\bold X}|$.

For a simplicial complex $\bold X$, denote by
$V({\bold X})$ and $E({\bold X})$ the {\it vertex set} and the
{\it edge set} of ${\bold X},$ namely, the set of all
$0$--dimensional and $1$--dimensional simplices of ${\bold X}.$ The pair
$(V({\bold X}),E({\bold X}))$ is called the {\it (underlying)
graph} or the {\it $1$--skeleton} of ${\bold X}$ and is denoted by
$G({\bold X})$. Conversely, for a graph $G$ one can derive a
simplicial complex ${\bold X}(G)$ (the {\it clique complex} of $G$)
by taking all complete
subgraphs (cliques)   as simplices of the complex.
A simplicial complex $\bold X$ is a {\it flag complex} (or a {\it
clique complex}) if any set of vertices is included in a face of
$\bold X$ whenever each pair of its vertices is contained in a
face of ${\bold X}$ (in the theory of hypergraphs this condition
is called conformality). A flag complex can therefore be recovered
by its underlying graph $G({\bold X})$: the complete subgraphs of
$G({\bold X})$ are exactly the simplices  of ${\bold X}.$
All simplicial complexes occurring in this paper are flag complexes.
The {\it link} of a simplex $\sigma$ in ${\bold X},$ denoted lk$(\sigma,{\bold X})$ is the
simplicial complex consisting of all simplices $\sigma'$ such that
$\sigma \cap \sigma'=\emptyset$ and $\sigma\cup\sigma'\in {\bold X}.$ For a
simplicial complex $\bold X$ and  a vertex $v$ not belonging to ${\bold X},$
the {\it cone} with apex $v$ and
base $\bold X$ is the simplicial
complex $v\ast {\bold X}={\bold X}\cup \{ \sigma\cup \{ v\}: \sigma\in {\bold X}\}.$

For a simplicial complex $\bold X$ and any $k\ge 1,$ the {\it Rips complex} ${\bold X}_k$
is a simplicial
complex with the same set of vertices as $\bold X$ and with a simplex spanned by any subset
$S\subseteq V({\bold X})$ such that $d(u,v)\le k$ in $G({\bold X})$ for each pair of vertices
$u,v\in S$ (i.e., $S$ has
diameter $\le k$ in the graph $G({\bold X})$); cf.\ e.g.\ \cite{G}. From the definition immediately follows that the
Rips complex of any complex is a flag complex.
Alternatively, the Rips complex ${\bold X}_k$ can be viewed as the clique complex
${\bold X}(G^k(\bold X))$ of the $k$th
power of the graph of $\bold X$ (the {\it $k$th power} $G^k$ of a graph $G$ has the same set of
vertices as $G$ and two vertices $u,v$ are adjacent in $G^k$ if and only if
$d(u,v)\le k$ in $G$).

\subsection{$SD_n$ property and weakly systolic complexes}

The following generalization of systolic complexes has been presented by Osajda \cite{Osajda}.
Let ${\bold X}$ be a flag simplicial complex and $\sigma^*$ be a simplex of ${\bold X}$. Then ${\bold X}$ satisfies the {\it $SD_n(\sigma^*)$ property} if for each $i\leq n$ and each simplex $\sigma$ located in the sphere $S_{i+1}(\sigma^*)$ the set $\sigma_0:=V(\rm{lk}(\sigma,{\bf X}))$$\cap B_i(\sigma^*)$
spans a non-empty simplex of $\bf X$ ($SD$ stands for {\it simple descent} on balls).  Systolic complexes are exactly the flag complexes which satisfy the $SD_n(\sigma^*)$ property for all
simplices $\sigma^*$ and all natural numbers $n$. On the other hand, the $5$--wheel $W_5$ (see the definition at the beginning of Section \ref{char}) is an example of a ($2$--dimensional) simplicial complex which
satisfies the $SD_1(\sigma^*)$ property for $\sigma^*$ being any vertex or triangle but not for $\sigma^*$ being a boundary edge. In view of this analogy and of subsequent results,
we define a {\it weakly systolic} complex to be a flag simplicial complex $\bold X$ which satisfies the $SD_n(v)$ property for all vertices
$v\in V({\bold X})$ and for all natural numbers $n$. We will also define a {\it weakly bridged} graph to be the underlying graph
of a weakly systolic complex.  It can be  shown (cf.\ Theorem \ref{weakly-systolic}) that $\bold X$ is a weakly systolic complex if  for each vertex $v$ and every $i$ it satisfies the
following two conditions:

\medskip\noindent
{\it Vertex condition} (V):  for every vertex
$w \in S_{i+1}(v),$ the intersection
$V(\rm{lk}$$(w,{\bf X}))\cap B_i(v)$ is a single simplex;

\medskip\noindent
{\it Edge condition} (E): for every
edge $e \in S_{i+1}(v),$ the intersection
$V(\rm{lk}$$(e,{\bf X}))\cap B_i(v)$ is nonempty.

In fact, this is the original definition of a weakly systolic complex given in \cite{Osajda}. Notice that these two conditions imply that weakly systolic complexes are exactly
the flag complexes whose underlying graphs are thin and satisfy the triangle condition.

\subsection{Dismantlability of graphs and LC-contractibility of complexes}
\label{dislc}

Let $G=(V,E)$ be a graph and $u,v$ two vertices of $G$ such that any neighbor of $v$ (including $v$ itself) is also a
neighbor of $u$, i.e.\ $B_1(v)\subseteq B_1(u)$.  Then there is a retraction of $G$ to $G-v$ taking $v$ to $u$. Following \cite{HeNe}, we call this retraction a {\it fold}
and we say that $v$ is {\it dominated} by $u.$ A finite graph $G$ is {\it dismantlable} if it can be reduced, by a sequence of folds, to a single vertex.
In other words, an $n$--vertex  graph $G=(V,E)$ is dismantlable if its vertices can be ordered $v_1,\ldots,v_n$ so that for each vertex $v_i, 1\le i<n,$
there exists another vertex $v_j$ with $j>i,$ such that $B_1(v_i)\cap V_i\subseteq B_1(v_j)\cap V_i,$ where $V_i:=\{ v_i,v_{i+1},\ldots,v_n\}.$ This order is called a
{\it dismantling order}.
We now consider the analogue of dismantlability for a simplicial complex $\bold X$ investigated in the papers  \cite{CiYa,Ma}.
A vertex $v$ of $\bold X$ is {\it LC-removable} if lk$(v,{\bold X})$ is a cone. If $v$ is an LC-removable vertex of $\bold X$,
then ${\bold X}-v:= \{ \sigma\in {\bold X}: v\notin \sigma\}$ is obtained from $\bold X$ by an {\it elementary LC-reduction} (link-cone reduction) \cite{Ma}. Then $\bold X$ is called
 {\it LC-contractible} \cite{CiYa}  if there is a sequence of elementary LC-reductions transforming $\bold X$ to one vertex.  For flag simplicial complexes,  the
 LC-contractibility of $\bold X$  is equivalent to dismantlability of its graph $G({\bold X})$ because an LC-removable  vertex $v$ is
 dominated by the apex of the cone  lk$(v,{\bold X})$ and vice versa the link of any dominated vertex $v$ is a cone having the vertex dominating $v$ as
 its apex. LC-contractible simplicial complexes are collapsible (see \cite[Corollary 6.5]{CiYa}).

The simplest algorithmic way to order the vertices of a locally finite graph is to apply the {\it Breadth-First Search} (BFS) starting from the root vertex
(base point) $b.$ We number with 1 the vertex $u$ and put it on the initially empty queue. We repeatedly remove the vertex $v$ at the head of the queue and consequently number (in an arbitrary order) and place onto the queue all still unnumbered neighbors of $v$. BFS constructs a spanning tree $T_u$ of $G$ with the vertex $u$ as a root. Then a vertex $v$ is the {\it father} in $T_v$ of any of its neighbors $w$ in $G$ included in the queue when $v$ is removed (notation $f(w)=v$).
Notice that the distance from any vertex $v$ to the root $u$ is the same in $G$ and in $T_u.$ Another method to order the vertices of a graph is the {\it Lexicographic Breadth-First Search} (LexBFS) proposed by Rose-Tarjan-Lueker \cite{RoTaLu}. According to LexBFS, the vertices of a graph $G$ are numbered in decreasing order. The {\it label} $L(w)$ of an unnumbered vertex $w$ is the list of its numbered neighbors. As the next vertex to be numbered, select the vertex with the lexicographic largest label, breaking ties arbitrarily. As in case of BFS, we remove the vertex $v$ at the head of the queue and consequently number according to the lexicographic order and place onto the queue all still unnumbered neighbors of $v$. LexBFS is a particular instance of BFS, i.e., every ordering produced by LexBFS can also be generated by BFS.

Anstee-Farber \cite{AnFa} established that bridged graphs are dismantlable. Chepoi \cite{Ch_bridged} noticed that any order of
a bridged graph returned by BFS is a dismantling order. Namely, he showed a stronger result: {\it for any two adjacent vertices $v_i,v_j$ with $i<j,$ their fathers $f(v_i),f(v_j)$ either coincide or
are adjacent and  moreover $f(v_j)$ is adjacent to $v_i$.}
Polat \cite{Po,Po1} defined dismantlability and BFS for arbitrary (not necessarily locally finite) graphs and extended the results of \cite{AnFa,Ch_bridged} to all bridged graphs.

\subsection{Group actions on simplicial complexes}

Let $G$ be a group acting by automorphisms on a simplicial complex $\bold X$.
By $\mr {Fix}_G{\bold X}$ we denote the \emph{fixed point set} of the action of $G$ on $\bf X$, i.e.\ $\mr {Fix}_G{\bf X}=\lk x\in {\bf X}|\; Gx=\lk x\rk \rk$.
Recall that the action is \emph{cocompact} if the orbit space $G\backslash {\bf X}$ is compact. The action of $G$ on a locally finite simplicial complex $\bf X$ is \emph{properly discontinuous} if stabilizers of simplices are finite. Finally, the action is \emph{geometric} (or $G$ \emph{acts geometrically} on $\bf X$) if it is cocompact and properly discontinuous.

\section{Characterizations of weakly systolic complexes}
\label{char}

We continue with the characterizations of weakly systolic complexes and their underlying graphs; some of those characterizations have been presented also in \cite{Osajda}. We denote by
$C_k$ a $k$--cycle and by $W_k$ a $k$--wheel, i.e., a $k$--cycle $x_1,\ldots,x_k$ plus a central vertex $c$ adjacent to all vertices of $C_k.$  $W_k$ can also be viewed as a $2$--dimensional
simplicial complex  consisting of $k$ triangles $\sigma_1,\ldots,\sigma_k$ sharing a common vertex $c$ and such that $\sigma_i$ and $\sigma_j$ intersect in an edge $x_ic$ exactly
when $|j-i|=1\; (\mr{mod}\; k).$ In other words,  lk$(c,{\bold W_k})=C_k$, i.e.\ $W_k$ is a cone over $C_k$. By $\widehat{W}_k$ we denote a $k$--wheel $W_k$ plus a triangle $ax_ix_{i+1}$ for some $i<k$ (we suppose that $a\ne c$ and that $a$ is not adjacent to any other vertex of $W_k$). We continue with a condition which basically characterizes
weakly systolic complexes among simply connected flag simplicial complexes:

\medskip\noindent
{\sf $\widehat{W}_5$--condition}: {\it for any $\widehat{W}_5,$ there exists a vertex $v\notin \widehat{W}_5$ such that $\widehat{W}_5$ is included in  $\rm{lk}(v,{\bold X}),$  i.e., $v$ is adjacent
in $G({\bold X})$ to all vertices of  $\widehat{W}_5$} (see Fig.~\ref{5-wheel}).

\begin{figure}[t]
\begin{center}
\scalebox{0.60}{\includegraphics{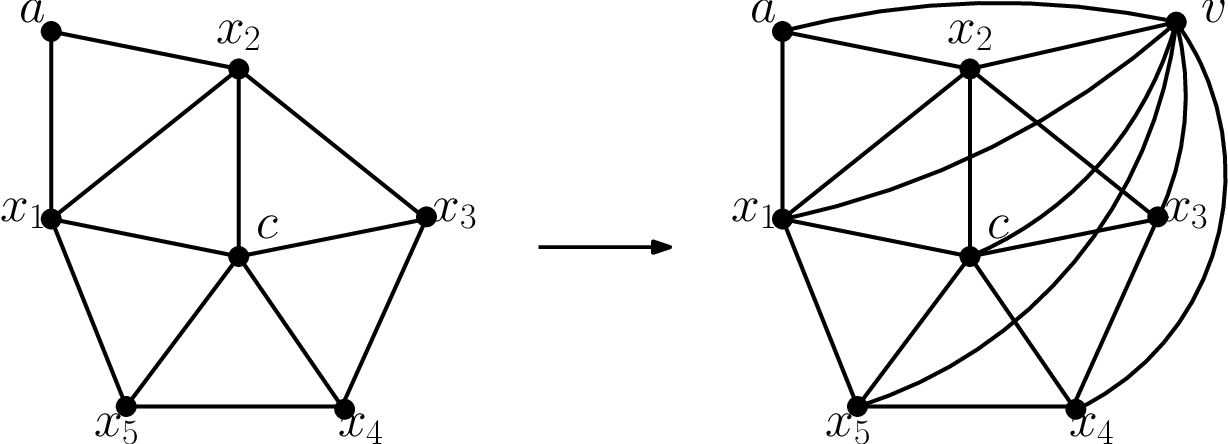}}
\end{center}
\caption{The $\widehat{W}_5$--condition}
\label{5-wheel}
\end{figure}

\begin{theorem}[Characterizations]
\label{weakly-systolic}
For a flag simplicial complex $\bold X$ the following conditions are equivalent:
\begin{itemize}
\item[(i)] ${\bold X}$ is weakly systolic;
\item[(ii)]  $\bold X$ satisfies the vertex condition (V) and the edge condition (E);
\item[(iii)] $G({\bold X})$ is a weakly modular thin graph;
\item[(iv)] $G({\bold X})$ is a weakly modular graph without induced $C_4$;
\item[(v)] $G({\bold X})$  is a weakly modular graph with convex balls;
\item[(vi)] $G({\bold X})$ is a graph with convex balls in which any $C_5$ is included in a $5$--wheel $W_5$;
\item[(vii)] $\bold X$ is simply connected,  satisfies the $\widehat{W}_5$--condition, and does not contain induced $C_4.$
\end{itemize}
\end{theorem}

\begin{proof}
First we show that the conditions (i) through (v) are equivalent and then we show that these conditions are equivalent to (vi) and to (vii).
The implications (i)$\Rightarrow$(ii) and (iii)$\Rightarrow$(iv) are obvious.

\medskip
(ii)$\Rightarrow$(iii): The condition (V)
implies that all vertices of $I(u,v)$ adjacent to $v$ are pairwise adjacent, i.e.,
that $G({\bold X})$ is thin. On the other hand, from the condition (E)
we conclude that if $1=d(v,w)<d(u,v)=d(u,w)=i+1,$ then $v$ and $w$
have a common neighbor $x$ in the  sphere $S_{i}(u),$ implying the triangle condition. Finally, in thin graphs the quadrangle condition is automatically
satisfied. This shows that  $G({\bold X})$ is a weakly modular thin graph.

\medskip
(iv)$\Rightarrow$(v): Let $B_i(u)$ be any ball in $G({\bold X})$.
Since $G({\bold X})$ is weakly modular and $B_i(u)$ is a connected subgraph,
to show that $B_i(u)$
is convex it suffices
to show that $B_i(u)$
is locally convex, i.e., {\it if $x,y\in B_i(u)$ and $d(x,y)=2,$ then $I(x,y)\subseteq B_k(u)$}; cf.\ \cite[Theorem 7(a)]{Ch_triangle}.
 Suppose by way of contradiction that $z\in I(x,y)\setminus B_i(u).$
Then necessarily $d(x,u)=d(y,u)=i$ and $d(z,u)=i+1.$ Applying the quadrangle condition, we infer that there exists a vertex $z'$ adjacent to $x$ and $y$ at
distance $i-1$ from $u.$ As a result, the vertices $x,z,y,z'$ induce a forbidden $4$--cycle, a contradiction.

\medskip
(v)$\Rightarrow$(i): Pick a simplex $\sigma$ in the sphere $S_{i+1}(u).$ Denote by $\sigma_0$ the set of all vertices $x\in S_i(u)$ such that
$\sigma\cup \{ x\}$ is a simplex of $\bold X$. Since the balls of $G({\bold X})$ are convex, necessarily
any two vertices of $\sigma_0$ are adjacent. Thus $\sigma_0$ and $\sigma\cup \sigma_0$
induce complete subgraphs of $G({\bold X}).$ Since $\bold X$ is a flag complex,  $\sigma_0$ and $\sigma\cup \sigma_0$ are simplices. Notice that
obviously $\sigma'\subseteq \sigma_0$ holds for any  other simplex $\sigma'\subseteq S_i(u)$ such that $\sigma\cup\sigma'\in {\bold X}.$ Therefore, to establish the $SD_i(u)$ property it remains to show
that $\sigma_0$ is non-empty.
This is obviously true if $\sigma$ is a vertex. Thus we suppose that $\sigma$ contains at least two vertices.
Let $x$ be a vertex of $S_i(u)$ which is adjacent to the maximum number of vertices of $\sigma.$ Since $G({\bold X})$ is weakly modular and
$\sigma$ is contained in $S_{i+1}(u)$, the vertex $x$ must be  adjacent to at least two vertices of $\sigma.$  Suppose by way of contradiction that $x$ is not adjacent
to a vertex $v\in \sigma.$ Pick any neighbor $w$ of $x$ in $\sigma.$ By the triangle condition, there exists a vertex $y\in S_i(u)$ adjacent to $v$ and $w.$ Since $w$
is adjacent to $x,y\in S_i(u)$ and $w\in S_{i+1}(u),$ the convexity of $B_i(u)$ implies that $x$ and $y$ are adjacent. Pick any other vertex $w'$ of
$\sigma$ adjacent to $x.$ Since $x$ is not adjacent to $v$ and $G({\bold X})$ does not contain induced $4$--cycles,
the vertices $y$ and $w'$ must be adjacent. Hence, $y$ is adjacent to $v\in \sigma$ and to all neighbors of $x$ in $\sigma,$ contrary to the choice of $x.$
Thus $x$ is adjacent to all vertices of $\sigma,$ i.e.,
$\sigma_0\ne \emptyset.$ This shows that $\bold X$ satisfies the $SD_n(u)$ property.

\medskip
(v)$\Rightarrow$(vi): Pick a $5$--cycle induced by the vertices $x_1,x_2,x_3,x_4,x_5.$ Since $d(x_4,x_1)=d(x_4,x_2)=2,$ by the triangle condition there exists a
vertex $y$ adjacent to $x_1,x_2,$ and $x_4.$ Since $G({\bold X})$ does not contain induced $4$--cycles, necessarily $y$ must be also adjacent to $x_3$ and $x_5,$
yielding a $5$--wheel.

\medskip
(vi)$\Rightarrow$(vii): To show that a flag complex ${\bold X}$ is
simply connected, it suffices to prove  that every simple cycle in the underlying graph of $\bold X$
is a modulo 2 sum of its triangular faces. Notice that the isometric cycles of an arbitrary graph $G$
constitute a basis of cycles of $G$. Since  $G({\bold X})$ is a graph with convex balls,
the isometric cycles of
$G({\bold X})$ have length 3 or 5 \cite{FaJa,SoCh}. By (vi), any $5$--cycle $C$ of $G({\bold X})$ extends to a $5$--wheel,
thus $C$ is a modulo 2 sum of triangles. Hence ${\bold X}$ is indeed simply connected. That $\bold X$ does not contain induced $4$--cycles and $4$--wheels follows from the convexity of balls.
Finally, pick an extended $5$--wheel $\widehat{W}_5:$ let $x_1,x_2,x_3,x_4,x_5$ be the vertices of the $5$--cycle, $c$ be the center of the $5$--wheel, and $x_1,x_2,a$ be the
vertices of the pendant triangle. Since $x_3$ and $x_5$ are not adjacent and the balls of $G({\bold X})$ are convex, necessarily $d(a,x_4)=2.$ Let $u$ be a common neighbor
of $a$ and $x_4.$ If $u$ is adjacent to one of the vertices $x_2$ and $x_3,$ then
to avoid induced $4$--cycles (forbidden by the convexity of balls in $G(\bf X)$), $u$ will be also adjacent to the second vertex and to $c$.
But if $u$ is adjacent to $c,$ then it will be adjacent to $x_1$ and therefore to $x_5$ as well. Hence, in this case $u$ will be adjacent to all
vertices $x_1,x_2,x_3,x_4,x_5,$ and $c,$ and we are done. So, we can suppose that $u$ is not adjacent to any one of the vertices $x_1,x_2,x_3,x_5,$ and $c.$ As a result, we obtain two
$5$--cycles induced by the vertices $a,x_2,x_3,x_4,u$ and $a,x_1,x_5,x_4,u.$ Each of these cycles extends to a $5$--wheel. Let $v$ be the center of the $5$--wheel extending the first cycle.
To avoid a $4$--cycle induced by the vertices $x_2,v,x_4,c,$ the vertices $v$ and $c$ must be adjacent. Subsequently, to avoid a $4$--cycle induced by the vertices $c,v,a,x_1,$ the vertices
$v$ and $x_1$ must be adjacent. Finally, to avoid a $4$--cycle induced by $x_1,v,x_4,x_5,$ the vertices $v$ and $x_5$ must be adjacent. In this way, we deduce that $v$ is adjacent
to all six vertices of  $\widehat{W}_5$, establishing the $\widehat{W}_5$--condition.
\medskip

(vii)$\Rightarrow$(iv): To prove this implication, as in \cite{Ch_CAT}, we will use minimal disk diagrams.
Let ${\mathcal D}$ and ${\bold X}$ be two simplicial complexes. A map
$\varphi:V({\mathcal D})\rightarrow V({\bold X})$ is called {\it simplicial} if $\varphi(\sigma)\in {\bold X}$ for
all $\sigma\in {\mathcal D}.$
If ${\mathcal D}$ is a planar triangulation (i.e.\ the $1$--skeleton of ${\mathcal D}$ is an embedded
planar graph whose all interior $2$--faces are triangles) and $C=\varphi(\partial {\mathcal D}),$
then $({\mathcal D},\varphi)$ is called  a {\it singular disk diagram} (or Van Kampen diagram)
for $C$ (for more details see \cite[Chapter V]{LySch}).
According to Van Kampen's lemma (\cite{LySch}, pp.150--151), for every
cycle $C$ of a simply connected simplicial complex  one can construct a
singular disk diagram. A singular disk diagram with no cut vertices (i.e.,
its $1$--skeleton is $2$--connected) is
called a {\it disk diagram.}  A {\it minimal (singular) disk} for $C$ is a
(singular) disk diagram ${\mathcal D}$ for $C$ with a minimum number of $2$--faces.
This number is called the {\it (combinatorial) area} of $C$ and is
denoted Area$(C).$ The  minimal disks diagrams $({\mathcal D},\varphi) $ of simple
cycles $C$ in $1$--skeletons of  simply connected simplicial complexes have the
following properties \cite{Ch_CAT}: (1) $\varphi$ bijectively maps
$\partial {\mathcal D}$ to $C$ and (2) the image of a $2$--simplex of $\mathcal D$ under
$\varphi$ is a $2$--simplex, and two adjacent $2$--simplices of $\mathcal D$ have distinct images under $\varphi.$

Let $C$ be a simple cycle in the underlying graph $G({\bold X})$ of a flag simplicial complex
$\bold X$ satisfying the condition (vii).

\medskip\noindent
{\bf Claim 1:} {\it If $C$ has length 5, then the minimal disk diagram for $C$ is a $5$--wheel. If the length of $C$ is not $5$, then $C$
admits a minimal disk diagram $\mathcal D$ which is a systolic complex, i.e., a plane triangulation
whose all inner vertices have degrees $\ge 6.$}

\medskip\noindent
{\bf Proof of Claim 1:}  First we show that any minimal disk diagram
$\mathcal D$ of $C$ does not contain interior vertices of degrees 3 and 4.
Let $x$ be any interior vertex  of $\mathcal D$.
Let $x_1,\ldots,x_k$ be the cyclically ordered neighbors of $x$
and let $\sigma_1,\sigma_2,\ldots, \sigma_k$ be the faces incident to $x$, where
$\sigma_i=xx_ix_{i+1(mod ~k)}$ $(i=1,\ldots,k)$. Trivially, $k\geq 3.$ Suppose by way of contradiction that $k\leq 4.$
By properties of minimal disk diagrams,  $\varphi(\sigma_1),\ldots,\varphi(\sigma_k)$ are
distinct $2$--simplices of ${\bold X}.$

\medskip\noindent
{\it Case 1:} $k=3$. Then the $2$--simplices  $\varphi(\sigma_1),\varphi(\sigma_2),
\varphi(\sigma_3)$ of $\bold X$ intersect in $\varphi(x)$ and pairwise share an edge of ${\bold X}.$ Since ${\bold X}$ is flag,
they are contained in a $3$--simplex of ${\bold X}.$ This implies that
$\delta=\varphi(x_1)\varphi(x_2)\varphi(x_3)$ is a $2$--face of ${\bold X}.$   Let ${\mathcal D}'$ be a
disk triangulation obtained from ${\mathcal D}$ by deleting the vertex $x$
and the triangles $\sigma_1, \sigma_2,\sigma_3,$ and  adding
the $2$--simplex $x_1x_2x_3.$  The map $\varphi: V({\mathcal D}')\rightarrow V({\bold X})$  is
simplicial, because it maps $x_1x_2x_3$  to  $\delta.$ Therefore
$({\mathcal D}',\varphi)$ is a  disk diagram for $C,$ contrary  to the minimality choice of ${\mathcal D}.$

\medskip\noindent
{\it Case 2:} $k=4$. Since two adjacent $2$--simplices of $\mathcal D$ have distinct images under $\varphi$, the cycle $C'=(x_1,x_2,x_3,x_4,x_1)$ is sent
to a $4$--cycle $\varphi(C')$ of lk$(\varphi(x),{\bold X})$. Since $G(\bf X)$ does not contain induced $4$--cycles, two opposite vertices of $\varphi(C')$, say $\varphi(x_1)$
and $\varphi(x_3),$ are adjacent. Consequently, since $\bf X$ is flag, $\delta'=\varphi(x_1)\varphi(x_3)\varphi(x_2)$
and $\delta''=\varphi(x_1)\varphi(x_3)\varphi(x_4)$ are $2$--faces of ${\bold X}.$ Let ${\mathcal D}'$ be a disk
triangulation obtained from ${\mathcal D}$ by deleting the vertex $x$ and the triangles
$\sigma_i (i=1,\ldots,4),$ and adding the $2$--simplices
$\sigma'=x_1x_3x_2$ and $\sigma''=x_1x_3x_4.$ The map $\varphi$ remains simplicial,
since it sends $\sigma',\sigma''$ to $\delta',\delta'',$ respectively,
contrary to the minimality choice of ${\mathcal D}.$
\medskip

This establishes that the degree of each interior vertex $x$ of any
minimal disk diagram is $\ge 5.$
Suppose now additionally that $\mathcal D$ is a minimal disk diagram for$C$ having a minimum
number of inner vertices of degree 5. We will denote the vertices of $\mathcal D$ and their images in $\bold X$
under $\varphi$ by the same symbols but specifying each time their position. Let $x$ be any interior vertex  of $\mathcal D$ of degree $5$ and
let $x_1,\ldots,x_5$ be the neighbors of $x.$ If
$C=(x_1,x_2,x_3,x_4,x_5,x_1)$ then we are done because $\mathcal D$ is a $5$--wheel. If
$C\neq (x_1,x_2,x_3,x_4,x_5,x_1)$ then one of the edges of the $5$--cycle
$(x_1,x_2,x_3,x_4,x_5,x_1),$ say $x_1x_2$, belongs in ${\mathcal D}$ to the second triangle $x_1x_2x_6.$
The minimality of $\mathcal D$ implies  that $x,x_1,x_2,x_3,x_4,x_5,x_6$ induce in $\bold X$ a
$\widehat{W}_5$ or that $x$ and $x_6$ are adjacent in $\bold X$. In the first case, by the $\widehat{W}_5$--condition,
there exists a vertex $y$ of $\bold X$ which is adjacent to all vertices of this $\widehat{W}_5.$
Let ${\mathcal D}'$ be a disk triangulation obtained from ${\mathcal D}$ by deleting the vertex $x$ and the five triangles incident to $x$ as well as the triangle $x_1x_2x_6$ and
replacing them by the six triangles of the resulting $6$--wheel centered at $y$ (we call this operation a {\it flip}). In the second case, let ${\mathcal D}'$ be a disk
triangulation obtained from ${\mathcal D}$ by deleting the the five triangles incident to $x$ as well as the triangle $x_1x_2x_6$ and
replacing them by the six triangles of the resulting $6$--wheel centered at $x.$ In both cases, the resulting map $\varphi$ remains simplicial.
${\mathcal D}'$ has the same number of triangles as $\mathcal D,$ therefore ${\mathcal D}'$ is also a minimal disk diagram for $C$.
The flip replaces in the first case the
vertex $x$ of degree 5 by the vertex $y$ of degree 6. In the second case, it increases  the degree of $x$ from 5 to 6. In both cases,  it preserves
the degrees of all other vertices except the vertices $x_1$ and $x_2,$ whose degrees
decrease by 1. Since, by the minimality choice
of $\mathcal D$, the disk diagram ${\mathcal D}'$ has at least as many inner vertices of degree 5 as ${\mathcal D},$ necessarily at least one of the vertices $x_1,x_2,$ say $x_1$,
is an inner vertex  of degree at most 6 of $\mathcal D$.   If the degree of $x_1$ in $\mathcal D$  is 5,
then in ${\mathcal D}'$ the degree of $x_1$ will be 4, which is impossible by what has been shown above because ${\mathcal D}'$ is also a minimal disk diagram and $x_1$
is an interior vertex of ${\mathcal D}'.$ Hence the degree of $x_1$ in $\mathcal D$ is 6 and its neighbors constitute an induced (in ${\mathcal D}$) $6$--cycle
$(x_6,x_2,x,x_5,u,v,x_6).$

\medskip\noindent
{\it Case 1:} $x$ and $x_6$ are not adjacent in $\bold X$. Since $\bold X$ does not contain induced $C_4$ and the minimal disk diagrams for $C$
do not contain interior vertices of degree 3 and 4, it can be easily shown that
the images in $\bold X$  of the vertices $x_5,y,x_6,v,u,x_1,x_4$ induce a $\widehat{W}_5$ constituted by the $5$--wheel centered at $x_1$ and the triangle $x_4yx_5.$
By the $\widehat{W}_5$--condition, there exists a vertex $z$ of $\bold X$ which is adjacent to all vertices of $\widehat{W}_5.$ If $z$ is adjacent in
$\bold X$ with all vertices of the $7$--cycle $(u,v,x_6,x_2,x_3,x_4,x_5,u),$ then replacing in $\mathcal D$ the 9 triangles incident to $x$ and $x_1$ by the 7 triangles
of $\bold X$  incident to $z,$ we will obtain a disk diagram ${\mathcal D}''$ for $C$ having less triangles than $\mathcal D$, contrary to the minimality of
$\mathcal D$. Therefore $z$ is different from $x$ and is not adjacent to one of the vertices $x_2,x_3.$ Since $x_1$ and $x_4$ are not adjacent and both $x$ and $z$ are adjacent to $x_1,x_4,$ to avoid an induced $C_4$ we conclude that $z$ is adjacent in $\bold X$ to $x.$ If $z$ is not adjacent to $x_2,$ then, since $x$ and $x_6$ are not adjacent, we will obtain a $C_4$ induced by
$x,z,x_6,x_2.$ Thus $z$ is adjacent to $x_2,$ and therefore $z$ is not adjacent to $x_3.$ Since both $z$ and $x_3$ are adjacent to nonadjacent vertices $x_2$ and $x_4,$  we will obtain a $C_4$ induced by $z,x_2,x_3,x_4.$ This contradiction shows that the degree of $x_1$ in $\mathcal D$ is at least 7.

\medskip\noindent
{\it Case 2:} $x$ and $x_6$ are adjacent in $\bold X$. Again, using the fact that the minimal disk diagrams for $C$ do not contain interior vertices of degree 3 and 4, the fact that $\bold X$ does not contain induced $C_4,$ it can be easily shown that $d(x,u)=2.$ Therefore  the vertices $x_1,x_2,x_3,x_4,x_5,x,u$ induce a $\widehat W_5$ constituted by
the $5$--wheel centered at $x$ and the triangle $x_1ux_5$. Thus, by the $\widehat W_5$--condition, there exists a vertex $y'\neq x$ containing $\widehat W_5$ in its link. Then considering the minimal disk diagram obtained by the flip exchanging $x$ and $y'$ we conclude that the vertices $u,v,x_6,x_2,y',x_1$ induce a $5$--wheel. Together with the vertex $x_3$ they induce
a $\widehat W_5$, so that, by $\widehat W_5$--condition, there exists a vertex $z$ adjacent to all the vertices $u,v,x_6,x_2,y',x_1, x_3$. If $z$ is adjacent to $x_4$ and $x_5$ then we get a disk diagram for $C$ having less triangles than $\mathcal D$, which contradicts the minimality of $\mathcal D$. If $z$ is not adjacent to one of the vertices $x_4,x_5$ then we also get a contradiction arguing as in Case 1. Therefore, in our case the degree of $x_1$ in $\mathcal D$ is also at least 7. This final contradiction shows that all interior vertices of $\mathcal D$ have degrees $\ge 6,$ establishing Claim 1.

\bigskip
From Claim 1 we deduce that any simple cycle $C$ of the underlying graph of $\bold X$ admits a minimal disk diagram $\mathcal D$ which is either
a $5$--wheel or a systolic plane triangulation.
We will refer to a degree two boundary vertex $v$ of $\mathcal D$ as a {\it corner of first type} and to a degree three boundary vertex $v$ of $\mathcal D$ as a {\it corner of second type}.
In the first case, the two neighbors of $v$ are adjacent. In the second
case, $v$ and its neighbors in $\partial {\mathcal D}$ are adjacent to the third
neighbor of $v.$
If ${\mathcal D}$ is a $5$--wheel then it has five corners of second type. Otherwise ${\mathcal D}$ is a systolic plane triangulation and we can use the Gauss-Bonnet formula ``sum over interior vertices of six minus degree
plus sum over boundary vertices of three minus degree equals six times Euler
characteristic"; see \cite[pp.\ 342--346]{LySch}. From this formula we infer that $\mathcal D$ contains
at least three corners, and if $\mathcal D$ has exactly three corners then they
are all of first type. Furthermore, if $\mathcal D$ contains exactly four corners, then at least two of them
are corners of first type.

\medskip\noindent
{\bf Claim 2:} {\it
$G({\bold X})$ is weakly modular, i.e.\ $G({\bold X})$ satisfies the triangle and quadrangle conditions.}

\medskip\noindent
{\bf Proof of Claim 2:}
To verify the triangle condition,
let $u,v,w$ be three vertices with
$1=d(v,w)\leq d(u,v)=d(u,w)=k.$ We claim that if
$I(u,v)\cap I(u,w)=\{ u\},$ then $k=1.$  Suppose not. Pick two shortest paths
$P'$ and $P''$ joining the pairs $u,v$ and $u,w,$ respectively, such that
the cycle  $C$ composed of $P',P''$ and the edge $vw$ has  minimal combinatorial area Area$(C)$ among all cycles constituted by the edge $vw$ and shortest paths connecting $u$ with $v$ and $w$ (the choice of $v,w$ implies that $C$ is a simple cycle).
Let $\mathcal D$ be a minimal disk diagram for $C$ satisfying Claim 1. Then either $\mathcal D$ has a corner  $x$
different from $u,v,w$ or the vertices $u,v,w$ are the only corners of $\mathcal D.$ In the second case, $u,v,w$ are all three corners of first
type, therefore the two neighbors of $v$ in $C$ will be adjacent. This means that $w$ will be adjacent to
the neighbor of $v$ in $P',$ contrary to $I(u,v)\cap I(u,w)=\{ u\}.$ Thus we can assume that a corner $x$ exists and $x$ is not one of $u,
v$ or $w$.
Without loss of generality we can assume $x$ is on the path $P'$. Let
$y$ and $z$ be its neighbors on $P'$. Note that $x$ cannot be of first type, since otherwise $y$ and $z$ are adjacent, contrary to
the assumption that $P'$ is a shortest path.
Thus $x$ is of the second type and there is a vertex $p$ of $\mathcal D$
adjacent to $x,y,z.$  If we replace in $P'$ the vertex $x$ by $p,$ we will
obtain a new shortest path between $u$ and $v.$  Together with $P''$ and
the edge $vw$ this path  forms a cycle $C'$ whose area is strictly smaller
than Area$(C),$ contrary to the choice of $C.$ This establishes the triangle condition.

To verify the quadrangle
condition, suppose by way of contradiction that we can find distinct
vertices $u,v,w,z$ such that $v,w\in I(u,z)$ are neighbors of $z$
and $I(u,v)\cap I(u,w)=\{u\},$ however $u$ is not adjacent to $v$ and $w.$
Again, select two shortest paths $P'$ and $P''$ between $u,v$ and $u,w,$
respectively, so that the cycle $C$ composed of $P',P''$ and the
edges $vz$ and $zw$ has minimum area.  Choose a minimal disk $\mathcal D$
of $C$ as in Claim 1. From the initial hypothesis concerning the vertices $u,v,w,z$
we deduce that $\mathcal D$ has at most one corner of first type located at $u.$
Hence $\mathcal D$ contains at least four corners of second type. Since one corner $x$ is distinct
from $u,v,w,z,$ then proceeding in the same way as in the triangle condition case, we will
obtain a contradiction with the choice of the paths $P',P''.$ This shows that $u$ is adjacent to $v,w,$
establishing the quadrangle condition. This concludes the proof of Claim 2.
\medskip

By Claim 2 $G({\bold X})$ is a weakly modular graph. On the other hand, by condition (vii) $G({\bold X})$ does not contain induced $C_4$. This
concludes the proof of the implication (vii)$\Rightarrow$(iv) and of the theorem.
\end{proof}

In the analysis of his construction of locally homogeneous graphs $H$ having a given regular graph  of girth $\ge 6$
(i.e., $6$--large) as a link of each vertex of $H$, Weetman \cite{Wee} introduced {\it quasitrees} as the graphs $G=(V,E)$
satisfying the following two conditions for each vertex $v:$ (F1) each vertex $x\in S_{i+1}(v)$ has one
or two adjacent neighbors in $S_i(v);$ (F2) any two adjacent vertices $x,y\in S_{i+1}(v)$ have a common neighbor $z\in S_i(v).$
It can be easily seen that (F2) is a reformulation of the edge condition (E) (alias the triangle condition). On the other hand,
(F1) is a particular case of the vertex condition (V). From Theorem \ref{weakly-systolic}(ii) we immediately obtain the following
observation:

\begin{corollary}\label{quasitrees} The simplicial complexes derived from quasitrees are weakly systolic. In particular, quasitrees
are weakly bridged graphs.
\end{corollary}

The $5$--wheel is an example of a quasitree which is not a bridged graph, thus not all simplicial complexes derived from quasitrees are systolic.

\section{Dismantlability of weakly bridged graphs}
\label{dismantlability}

In this section, we show that the underlying graphs of weakly systolic complexes are dismantlable
and that a dismantling order can be obtained using LexBFS. Then we use this result to deduce
several consequences about collapsibility of weakly
systolic complexes and fixed simplices. Other consequences of dismantling are given in subsequent sections.

\begin{theorem}[LexBFS dismantlability]
\label{dismantle}
Any LexBFS ordering of a locally finite weakly bridged graph $G$ is a dismantling ordering.
\end{theorem}

\begin{proof} We will establish the result for finite weakly bridged graphs.
The proof in the locally finite case is completely similar. Let $v_n,\ldots,v_1$ be the total order returned by the LexBFS starting from
the basepoint $u=v_n.$ Let $G_i$ be the subgraph of $G$ induced by the vertices $v_n,\ldots,v_i.$
For a vertex $v\ne u$ of $G,$ denote by $f(v)$ its father in the LexBFS
tree $T_u,$ by $L(v)$ the list of all neighbors of $v$ labeled before $v,$ and by $\alpha(v)$
the label of $v$ (i.e., if $v=v_i,$ then $\alpha(v)=i$). We decompose the label $L(v)$ of each vertex $v$ into two parts $L'(v)$ and $L''(v):$
if $d(v,u)=i,$ then $L'(v)=L(v)\cap S_{i-1}(u)$ and $L''(v)=L(v)\cap S_i(u).$  Notice that in the lexicographic order of $L(v),$ all vertices
of $L'(v)$ precede the vertices of $L''(v);$ in particular, the father of $v$ belongs to $L'(v).$
The proof of the theorem is a consequence of the following assertion, which we call the {\it fellow traveler property}:

\medskip\noindent
{\bf Fellow Traveler Property:} {\it If $v,w$ are adjacent vertices of $G,$ then their fathers $v'=f(v)$ and $w'=f(w)$ either coincide or are adjacent. If $v'$ and $w'$ are adjacent and
$\alpha(w)<\alpha(v),$ then $w'$ is adjacent to $v$ and $v'$ is not adjacent to $w.$}

\medskip
Indeed, if this assertion holds, then  we claim that $v_n,\ldots,v_1$ is a dismantling order. To see this, it suffices to show
that any vertex $v_i$ is dominated in $G_i$ by its father $f(v_i)$ in the LexBFS tree $T_u.$ Pick any neighbor $v_j$ of $v_i$
in $G_i.$ We assert that $v_j$ coincides or is adjacent to $f(v_i).$ This is obviously true if $f(v_j)=f(v_i).$ Otherwise, if
$f(v_i)\ne f(v_j),$ then the Fellow Traveler Property implies that $f(v_i)$ and $f(v_j)$ are adjacent and since $i<j$ that $v_j$ is adjacent to $f(v_i).$
This shows that indeed any LexBFS order is a dismantling order.

Therefore, it remains to prove the Fellow Traveler Property which we establish in the following lemma.

\begin{lemma}
\label{feltra}
$G$ satisfies the Fellow Traveler Property.
\end{lemma}

\noindent
{\it Proof of Lemma \ref{feltra}.}
We proceed by induction on $i+1:=\mbox{max}\{ d(u,v),d(u,w)\}.$
We distinguish two cases: $d(u,v)<d(u,w)$ and $d(u,v)=d(u,w)=i+1$.
\medskip

\noindent
{\it Case 1: $d(u,v)<d(u,w).$}
Then $v,w'\in I(w,u)$ and since $G$ is thin, we conclude that $v$ and $w'$ either coincide or are adjacent. In the first case we are done because
$v$ (and therefore $w'$) is adjacent to its father $v'=f(v)$. If $v$ and $w'$ are adjacent, since $i=d(u,v)=d(u,w'),$
the vertices $v'$ and $f(w')$ coincide or are adjacent  by the induction assumption. Again, if $v'=f(w'),$ the assertion is immediate. Now suppose that $v'$ and $f(w')$
are adjacent. Since $w'=f(w)$ was labeled before $v$ (otherwise the father of $w$ is $v$ and not $w'$), $f(w')$ must be labeled before $v',$
therefore by the induction hypothesis we deduce that $v'=f(v)$ must be adjacent to $w'=f(w).$ This concludes the analysis of the case
$d(u,v)<d(u,w).$
\medskip

\noindent
{\it Case 2: $d(u,v)=d(u,w)=i+1.$}
Suppose, without loss of generality that $\alpha(w)<\alpha(v).$ If the vertices $v'=f(v)$ and $w'=f(w)$ coincide, then we are done. If the vertices
$v'$ and $w'$ are adjacent, then  the vertices $v,w,w',v'$ define a $4$--cycle. Since $G$ is weakly bridged, by Theorem \ref{weakly-systolic} this cycle cannot be induced. Since
$v$ was labeled before $w,$ the vertex $v'$ must be labeled before $w'.$ Therefore, if $v'$ is adjacent to $w,$ then LexBFS will label $w$ from $v'$
and not from $w',$ giving a contradiction. Thus $v'$ and $w$ are not adjacent, showing that $w'$ must be adjacent to $v,$ establishing the required assertion. So, assume
by way of contradiction that the vertices $v'$ and $w'$ are not adjacent in $G.$ Then $w'$ is not adjacent to $v,$ otherwise $w',v'\in B_i(u)$ and
$v\in I(v',w')\cap S_{i+1}(u),$ contrary to the convexity of the ball
$B_i(u)$ (similarly, $v'$ is  not adjacent to $w$).

Since $G$ is weakly modular by Theorem \ref{weakly-systolic}(iii), the
triangle condition applied to the vertices $v,w,$ and $u$ implies that there
exists a common neighbor $s$ of $v$ and $w$ located at distance $i$
from $u.$ Denote by $S$ the set of all such vertices $s.$ From the
property $SD_i(u)$ we infer that $S$ is a simplex of $\bold X$ (i.e.,
its vertices are pairwise adjacent in $G$). Since $v'$ and $w'$ do
not belong to $S,$ necessarily all vertices of $S$ have been labeled
later than $v'$ and $w'$ (but obviously before $v$ and $w$). Pick a
vertex $s$ in $S$ with the largest label $\alpha(s)$ and set
$z:=f(s).$ By the induction assumption applied to the pairs of adjacent
vertices $\{ v',s\}$ and $\{ s,w'\}$, we conclude that the vertices
of each of the pairs $\{ f(v'),z\}$ and $\{ z,f(w')\}$ either
coincide or are adjacent. Moreover, in all cases, the vertex $z$
must be adjacent to the vertices $v'$ and $w'.$

\medskip\noindent
{\bf Claim 1:} $L'(v')=L'(s)=L'(w')$ and $z$ is the father of
$v'$ and $w'.$

\medskip\noindent
{\bf Proof of Claim 1:} Since $s$ was labeled later than $v'$ and
$w',$ it suffices to show that $L'(v')=L'(s).$
Indeed, if this is the case, then necessarily $z$ is
the father of $v'.$  Then, as $z$ is adjacent to $w'$ and $\alpha(w')<\alpha(v'),$ necessarily $z$ is also
the father of $w'.$ Now, if $L'(w')$ and $L'(s)=L'(v')$ do
not coincide, since $L'(v')$  lexicographically precedes  $L''(v')$ and $L'(w')$ precedes $L''(w'),$
the fact that LexBFS labeled $v'$ before $w'$ means that
$L'(v')$ lexicographically precedes $L'(w').$  Since $L'(s)=L'(v'),$ then necessarily LexBFS would label
$s$ before $w',$ a contradiction. This shows that the equality of the two labels $L'(s)$ and $L'(v')$
implies the equality of the three labels $L'(v'),L'(s),$ and $L'(w').$

To show that $L'(v')=L'(s),$ since $\alpha(s)<\alpha(v'),$ it suffices to establish only the inclusion
$L'(v')\subseteq L'(s).$ Suppose by way of contradiction that
there exists a vertex in $L'(v')\setminus L'(s)$ i.e., a vertex $x\in S_{i-1}(u)$ which is adjacent to
$v'$ but is not adjacent to $s.$ Let $x$ be the vertex of
$L'(v')\setminus L'(s)$ having the largest label $\alpha(x).$ Since $s$ was labeled by LexBFS later than
$v',$ necessarily any vertex of $L'(s)\setminus L'(v')$
must be labeled later than $x.$  Notice that $x$ cannot be adjacent to $w',$ since otherwise we would obtain an
induced $4$--cycle formed by the vertices $v',s,w',x.$
On the other hand $x$ is adjacent to $z$ because both vertices belong to the convex ball $B_{i-1}(u)$ and both are adjacent to the vertex $v'\in S_i(u)$.
Since $x$ is not adjacent to $v,w,$ and $s,$ we conclude
that the vertices $v,w,w',z,c',s,x$ induce an
extended $5$--wheel $\widehat{W}_5$. By the $\widehat{W}_5$--condition, there exists a vertex $t$ adjacent to all vertices of this
$\widehat{W}_5.$ Notice that $t\in S_i(u)$ because it is adjacent to $x\in S_{i-1}(u)$ and $v\in S_{i+1}(u)$. Hence $t\in S.$
Further $t$
must be adjacent to $z,$ otherwise we obtain a forbidden $4$--cycle induced by the vertices
$s,z,x,$ and $t.$ For the same reason, $t$ must be adjacent to any other vertex $z'$ belonging to
$L'(v')\cap L'(s).$ This means that LexBFS will label $t$ before $s.$ Since $t$ belongs to $S$ and
$\alpha(t)>\alpha (s),$ we obtain a contradiction with the choice of the vertex $s.$ This contradiction
concludes the proof of the Claim 1.

\medskip
We continue with the analysis of Case 2. Since $v'$ and $w'$ are not adjacent and $G$ does not contain
induced $4$--cycles, any vertex $s'\ne s$ adjacent to $v'$ and $w'$ is
also adjacent to $s.$ In particular, this shows that $L''(v')\cap
L''(w')\subseteq L''(s).$ Therefore, if $L''(w')\subseteq L''(v'),$
then $L''(w')\subseteq L''(s).$ Since $v'\in L''(s)\setminus
L''(w')$ and $L'(s)=L'(w')$ by Claim 1, we conclude that the vertex
$s$ must be labeled before $w',$ contrary to the assumption that
$\alpha(s)<\alpha(w').$ Therefore the set $B:=L''(w')\setminus
L''(v')$  is nonempty. Since $v'$ was labeled before $w'$ and
$L'(v')=L'(w')$ by Claim 1, we conclude that the set
$A:=L''(v')\setminus L''(w')$ is nonempty as well. Let $p$ be the
vertex of $A$ with the largest label $\alpha(p)$ and let $q$ be the
vertex of $B$ with the largest label $\alpha(q).$ Since LexBFS labeled
$v'$ before $w'$ and $L'(v')=L'(w')$ holds, necessarily
$\alpha(q)<\alpha(p)$ holds. Since $p\in L''(v'),$ we obtain that
$\alpha(w')<\alpha(v')<\alpha(p).$ Since $v'=f(v)$ and $w'=f(w),$
this shows that $p$ cannot be adjacent to the vertices $v$ and $w.$
If $s$ is adjacent to $p,$ then $p\in L''(s).$ But then from Claim 1
and the inclusion $L''(v')\cap L''(w')\subseteq L''(s)$ we could
infer that LexBFS must label $s$ before $w',$ contrary to the
assumption that $\alpha(s)<\alpha(w').$ Therefore $p$ is not adjacent to
$s$ either. On the other hand, since $\alpha(v')<\alpha(p),$ by the
induction hypothesis applied to the adjacent vertices $p$ and $v',$
we infer that $z=f(v')$ must be adjacent to $p.$ Hence the vertices $v,w,w',z,v',s,p$
induce an extended $5$--wheel. By the $\widehat{W}_5$--condition, there exists a vertex $t$ adjacent to all
these vertices.  Since $C:=L'(v')=L'(w')$ and
$d(v',w')=2,$ to avoid induced $4$--cycles, the vertex $t$ must be
adjacent to any vertex of $C.$ For the same reason, $t$ must be
adjacent to any vertex of $L''(v')\cap L''(w').$ Since additionally
$t$ is adjacent to the vertex $p$ of $A$ with the highest label,
necessarily $t$ will be labeled by LexBFS before $w'$ and $s.$ Since
$t$ is adjacent to $v$ and $w,$ this contradicts the assumption that
$w'=f(w).$ This shows that the initial assumption that $v'$ and
$w'$ are not adjacent lead to a final contradiction. Hence the order returned
by LexBFS is indeed a dismantling order
of the weakly bridged graph $G.$
This completes the proof of the lemma and of the theorem.
\end{proof}

\begin{corollary}
\label{discplx}
Locally finite weakly systolic complex ${\bold X}$ and every its Rips complex ${\bold X}_k$ are LC-contractible
and therefore collapsible.
\end{corollary}
\begin{proof}
Again we consider only the finite case. To show that any finite  weakly systolic complex ${\bold X}$ is LC-contractible it
suffices to notice that, since ${\bold X}$ is a flag complex, the LC-contractibility of ${\bold X}$
is equivalent to the dismantlability of its graph $G=G({\bold X}),$ and hence the result follows from Theorem \ref{dismantle}.

To show that the Rips complex ${\bold X}_k$ is LC-contractible,
since ${\bold X}_k$ is a flag complex, it suffices to show that its graph $G({\bold X}_k)$ is dismantlable.
From the definition of ${\bold X}_k,$ the graph $G({\bold X}_k)$ coincides with the $k$th power $G^k$ of the underlying
graph $G$ of $\bold X$. Now notice that if a vertex $v$ is dominated in $G$ by a vertex $u,$ then $u$ also dominates
$v$ in the graph $G^k.$ Indeed, pick any vertex $x$ adjacent to $v$ in $G^k.$ Then $d(v,x)\le k$ in $G.$ Let $y$ be the neighbor
of $v$ on some shortest path $P$ connecting the vertices $v$ and $x$ in $G.$ Since $u$ dominates $v,$ necessarily $u$ is adjacent to
$y.$ Hence $d(u,x)\le k$ in $G,$ therefore $u$ is adjacent to $x$ in $G^k.$ This shows that $v$ is dominated by $u$
in the graph $G^k$ as well. Therefore the dismantling order of $G$ returned by LexBFS is also a dismantling order of $G^k,$ establishing
that the Rips complex ${\bold X}_k$ is LC-contractible.
\end{proof}
\begin{corollary} \label{Rips} Graphs of Rips complexes ${\bold X}_n$ of
locally finite systolic and weakly systolic complexes are dismantlable.
\end{corollary}

For a locally finite weakly bridged graph $G$ and integer $k$ denote by $G_k$ the subgraph of $G$
induced by the first $k$ labeled vertices in a LexBFS order, i.e., by the vertices of $G$ with $k$ lexicographically largest labels.

\begin{corollary}\label{G_k} Any $G_k$ is an isometric  weakly bridged subgraph of $G.$
\end{corollary}

\begin{proof} By Theorem \ref{dismantle}, LexBFS returns  a dismantling order of $G$, hence any $G_k$ is
an isometric subgraph of $G.$ Therefore $G_k$ is a thin graph, because any interval $I(x,y)$ in $G_k$ is contained in the
interval of $G$ between $x$ and $y.$ Moreover, $G_k,$ as an isometric subgraph of a $G$, does not contain isometric cycles of length $>5.$
Hence, by a result of \cite{SoCh,FaJa}, $G_k$ is a graph with convex balls. By Theorem \ref{weakly-systolic}(vi) it remains to show that any induced $5$--cycle
$C$ of $G_k$ is included in a $5$--wheel. Suppose by the induction assumption that this is true for $G_{k-1}$. Therefore $C$ must contain the last labeled
vertex of $G_k$. Denote this vertex by $v$ and let $x,y$ be the neighbors of $v$ in $C.$ Let $v'=f(v)$ be the vertex (of $G_k$) dominating $v$ in $G_k.$
Since $C$ is induced, necessarily $v'$ is adjacent to $x$ and $y$ but distinct from these vertices.  Denote by $C'$ the $5$--cycle obtained by replacing in
$C$ the vertex $v$ by $v'.$  If $C'$ is not induced, then $v'$ will be adjacent to a third vertex of $C,$ and since $G_k$ does not contain induced $4$--cycles,
$v'$ will be adjacent to all vertices of $C,$ showing that $C$ extends to a $5$--wheel. So, suppose that $C'$ is induced. Applying  the induction hypothesis to $G_{k-1},$ we conclude that
$C'$ extends to a $5$--wheel in $G_{k-1}.$ Let $w$ be the central vertex of this wheel. To avoid a $4$--cycle induced by the vertices $x,y,v,$ and $w,$ necessarily $v$ and
$w$ must be adjacent. Hence $C$ extends in $G_k$ to a $5$--wheel centered at $w.$ This establishes that $G_k$ is indeed weakly bridged.
\end{proof}

A {\it simplicial map} on a simplicial complex
${\bold X}$  is a map
$\varphi: V({\bold X})\rightarrow V({\bold X})$ such that for all $\sigma\in {\bold X}$ we have $\varphi({\sigma})\in {\bold X}.$
A {\it homomorphism} of a graph $G=(V,E)$ is a simplicial map on a one-dimensional simplicial complex $G$.
A simplicial map fixes a simplex
$\sigma\in {\bold X}$ if $\varphi (\sigma)=\sigma$. Every simplicial map on $\bold X$ is a homomorphism of its graph $G({\bold X}).$  Every
homomorphism of a graph $G$ is a simplicial map on its clique complex ${\bold X}(G).$
Therefore, if ${\bold X}$ is a flag complex, then the set of simplicial maps of ${\bold X}$ coincides with the set of homomorphisms
of its graph $G({\bold X}).$ It is well know (see, for example, \cite[Theorem 2.65]{HeNe}) that any homomorphism of a finite dismantlable graph to
itself fixes some clique. From Theorem \ref{dismantle}
we know that the graphs of weakly systolic complexes as well as the graphs of their Rips
complexes are dismantlable. Therefore from the preceding discussion we obtain:

\begin{corollary}\label{fixed-clique}
Let $\bf X$ be a finite weakly systolic complex. Then any simplicial map of $\bold X$ to itself or of its Rips complex ${\bold X}_k$ to itself
fixes some simplex of the respective complex.
Any homomorphism of $G=G({\bold X})$ to itself
fixes some clique.
\end{corollary}

\section{Fixed point theorem}
\label{fixedpt}
In this section, we establish the fixed point theorem (Theorem C from Introduction).
We start with two auxiliary results. The first is an easy
corollary of Theorem \ref{dismantle}:

\begin{lemma}[Strictly dominated vertex]
\label{str dominat}
Let $\bold X$ be a finite weakly systolic complex.
Then either $\bold X$ is a single simplex or it contains two vertices
$v,w$ such that $B_1(v)$ is a proper subset of $B_1(w)$, i.e. $B_1(v)\subsetneq B_1(w)$ .
\end{lemma}

\begin{proof} Let $v$ be the last vertex of $\bold X$ labeled  by LexBFS which started at vertex $u$
(see Theorem \ref{dismantle}). If $d(u,v)=1,$ then the construction of our ordering implies that $B_1(u)=V({\bold X})$.
Hence, either there exists a vertex $w$, such that $B_1(w)\subsetneq
V({\bold X})=B_1(u),$ and we are done,  or every two vertices of $\bold X$ are adjacent,  i.e., $\bold X$ is a simplex.
Now suppose that $d(u,v)\geq 2$. Let $w$ be the father of $v$ and let $z$ be the father of $w$. From Theorem \ref{dismantle} we know that $B_1(v)\subseteq B_1(w)$. Since $d(u,v)=d(u,w)+1\geq 2$, we conclude that $u\neq w$ and that $z\in B_1(w)\setminus B_1(v)$.
Hence $B_1(v)$ is a proper subset of $B_1(w).$
\end{proof}

\begin{lemma}[Elementary LC-reduction]
\label{LC}
Let $\bold X$ be a finite weakly systolic complex. Let
$v,w$ be two vertices
such that $B_1(v)$ is a proper subset of $B_1(w)$. Then the full subcomplex ${\bold X}_0$ of $\bf X$
spanned by all vertices of $\bold X$ except $v$ is weakly systolic.
\end{lemma}
\begin{proof}
It is easy to see that ${\bf X}_0$ is simply connected (see also the discussion in Section \ref{dislc}).
Thus, by condition (vii) of Theorem \ref{weakly-systolic}, it suffices to show that ${\bf X}_0$ does not contain
induced $4$--cycles and satisfies the $\widehat{W}_5$--condition. Since, by Theorem \ref{weakly-systolic},
$\bold X$ does not contain induced $C_4,$ the same is true for its
full subcomplex ${\bf X}_0$. Let $\widehat{W}_5\subseteq {\bf X}_0$ be a given $5$--wheel plus a triangle as defined in Section \ref{char}.
By Theorem \ref{weakly-systolic} there exists a vertex $v'\in {\bold X}$ adjacent in $\bold X$ to all vertices of $\widehat{W}_5$.
If $v'\neq v$ then $v'\in {\bf X}_0$ and if $v'=v$ then
$\widehat{W}_5 \subseteq \mr{lk}(w,{\bf X}_0)$. In both cases
all vertices of $\widehat{W}_5$ are adjacent to a vertex of ${\bf X}_0$: $\widehat{W}_5$ is coned to $v$ in one case and to
$w$ in the other. Thus $\bold X_0$ also
satisfies the $\widehat{W}_5$--condition and hence the lemma follows.
\end{proof}

\begin{theorem}[The fixed point theorem]
\label{fpt}
Let $G$ be a finite group acting by simplicial automorphisms
on a locally finite weakly systolic complex ${\bold X}$. Then there exists
a simplex $\sigma \in {\bold X}$ which is invariant under the action of $G$.
\end{theorem}

\begin{proof}
Let ${\bold X}'$ be the subcomplex of $\bold X$ spanned by the convex hull of the set $Gz=\{ gz:\; \; g\in G\}$, for an arbitrary vertex $z$. Since $Gz$ is finite and, by Theorem \ref{weakly-systolic}(v), balls in $\bf X$ are convex, $\bf X'$ is a bounded full subcomplex of $\bf X$. Since $\bf X$ is locally finite, $\bf X'$ is finite.
Moreover, as a convex subcomplex of a weakly systolic complex,  ${\bold X}'$ is itself weakly systolic. Clearly ${\bold X}'$ is also $G$--invariant.
Thus there exists a minimal finite non-empty $G$--invariant
subcomplex ${\bold X}_0$ of $\bold X$, that is itself weakly systolic.
We assert that ${\bold X}_0$ must be a single simplex.

Assume by way of contradiction that ${\bold X}_0$ is not a simplex. Then, by Lemma \ref{str dominat}, ${\bf X}_0$ contains two vertices
$v,w$ such that $B_1(v)\subsetneq B_1(w)$ (i.e., $v$ is a strictly dominated vertex). Since the strict inclusion of $1$--balls
is a transitive relation and ${\bold X}_0$ is finite, there exists a finite set $S$ of strictly dominated vertices of $\bold X_0$
with the following property: for a vertex $x\in S$
there is no vertex $y$ with $B_1(y)\subsetneq B_1(x)$. Let ${\bold X}_0'$ be the full subcomplex of $\bf X$ spanned
by $V({\bold X}_0)\setminus S$. It is clear that ${\bold X}_0'$ is a
non-empty $G$--invariant proper subcomplex of ${\bold X}_0$. By Lemma \ref{LC}, ${\bold X}_0'$ is weakly systolic. This contradicts
the minimality of ${\bold X}_0$ and thus shows that ${\bold X}_0$ has to be a simplex.
\end{proof}

\begin{corollary}[Conjugacy classes of finite subgroups]
\label{conj}
Let $G$ be a group acting geometrically by automorphisms on a weakly systolic complex $\bf X$ (i.e. $G$ is weakly systolic).
Then $G$ contains only finitely many conjugacy classes of finite subgroups.
\end{corollary}

\begin{proof}
Suppose by way of contradiction that we have infinitely many conjugacy classes of finite subgroups represented by $H_1,H_2,\ldots\subseteq G$.
Since $G$ acts geometrically on ${\bf X},$ there exists a compact subset $K\subseteq V({\bf X})$ with $\bigcup_{g\in G} gK={\bf X}$. For $i=1,2,\ldots,$
let $\sigma_i$ be an $H_i$--invariant simplex of $\bf X$ (whose existence is assured by the fixed point Theorem \ref{fpt}) and let $g_i\in G$ be such
that $g_i(\sigma_i)\cap K\neq \emptyset$.
Then $g_i(\sigma_i)$ is $g_iH_ig_i^{-1}$ invariant and
$\bigcup_i g_iH_ig_i^{-1}$ is infinite. But for every element $g\in \bigcup_i g_iH_ig_i^{-1}$ we have $g(B_1(K))\cap B_1(K)\neq \emptyset,$ a
contradiction with the properness of the $G$--action on $\bf X$.
\end{proof}

\section{Contractibility of the fixed point set}
\label{contrfps}

The aim of this section is to prove that for a group acting on a
weakly systolic complex its fixed point set is
contractible (Proposition
\ref{inv set contr}). As explained in the Introduction, this result
implies Theorem E asserting that weakly systolic complexes are models
for $\eg$ for groups acting on them properly.

Our proof closely follows Przytycki's proof of an analogous result for
the case of systolic complexes \cite{Pr3}. There are however minor
technical difficulties. In particular, since balls around simplices in weakly systolic
complexes need not to be convex, we have to work with other convex objects that are
defined as follows. For a simplex  $\sigma$  of a simplicial complex $\bf X,$ set
$K_0(\sigma)=\sigma$ and $K_i(\sigma)=\bigcap _{v\in \sigma} B_i(v)$ for $i=1,2,\ldots$.

\begin{lemma}[Properties of $K_i(\sigma)$]
\label{bint}
 Let $\sigma$ be a simplex of a weakly systolic complex $\bf X$.
Then, for $i=0,1,2,...$, $K_i(\sigma)$ is convex and
$K_{i+1}(\sigma)\subseteq B_1(K_i(\sigma))$.
\end{lemma}

\begin{proof} Trivially, $K_0(\sigma)=\sigma$ is convex. For $i>0,$ $K_i(\sigma)$ is
the intersection of the balls $B_i(v), v\in \sigma.$ By Theorem \ref{weakly-systolic},
balls around vertices are convex, whence  $K_i(\sigma)$ is convex as well.
To establish the inclusion $K_{i+1}(\sigma)\subseteq B_1(K_i(\sigma)),$ pick any vertex $w\in K_{i+1}(\sigma).$
Let $l=d(w,\sigma)-1$ and
denote by $\sigma_0$ the metric projection of $w$ in $\sigma$.  By the property $SD_{l}(w),$ there exists a vertex $z\in S_{l}(w)$
adjacent to all vertices of the simplex $\sigma_0.$ Let $w'$ be a neighbor of $w$ in the interval $I(w,z).$ Then obviously $d(w',\sigma)=l$
and therefore $\sigma_0$ is the metric projection of $w'$ in $\sigma.$ Since $d(w',v)=d(w,v)-1$ for any vertex $v\in \sigma$ and $w\in K_{i+1}(\sigma),$
we conclude that $w'\in K_i(\sigma),$ whence $w\in B_1(w')\subseteq B_1(K_i(\sigma)).$
\end{proof}

We recall now two general results that were proved in
\cite{Pr3} and which will be important in the proof of Proposition
\ref{inv set contr}.

\begin{proposition}[{\cite[Proposition 4.1]{Pr3}}]
\label{p4.1p3}
If $\cal C, \cal D$ are posets and $F_0,F_1\colon \cal C \to \cal D$
are functors such that for each object $c$ of $\cal C$ we have
$F_0(c) \leq F_1(c)$, then the maps induced by $F_0$, $F_1$ on the
geometric realizations of $\cal C,\cal D$ are homotopic. Moreover
this homotopy can be chosen to be constant on the geometric
realization of the subposet of $\cal C$ of objects on which $F_0$ and
$F_1$ agree.
\end{proposition}

\begin{proposition}[{\cite[Proposition 4.2]{Pr3}}]
\label{p4.2p3}
Let $F_0\colon \cal C' \to \cal C$ be the functor from the flag
poset $\cal C'$ of a poset $\cal C$ into the poset $\cal C$,
assigning to each object of $\cal C'$, which is a chain of objects of
$\cal C$, its minimal element. Then the map induced by $F_0$ on
geometric realizations of $\cal C',\cal C$ (that are homeomorphic in
a canonical way) is homotopic to identity.
\end{proposition}

The following property of flag complexes  will be crucial in the definition of expansion by projection below.
It says that in weakly systolic case we can define projections on convex subcomplexes the same way as projections on balls.

\begin{lemma}[Projections on convexes]
\label{proj}
Let ${\bf X}$ be a simplicial flag complex and let $Y$ be its convex subset. If a simplex $\sigma$ belongs to $S_1(Y),$ i.e.
$\sigma \subseteq B_1(Y)$ and $\sigma \cap Y=\emptyset,$ then $\tau:=\mr{lk}(\sigma, {\bf X})\cap Y$ is a single simplex.
\end{lemma}

\begin{proof} By definition of links, $\tau$ consists of all vertices $v$ of $Y$ adjacent in $G({\bf X})$
to all vertices of $\sigma$. Since the set $Y$ is convex and $\sigma$ is disjoint from $Y,$ necessarily the vertices of $\tau$  are
pairwise adjacent. As $\bf X$ is a flag complex, $\tau$ is a simplex of $\bf X$.
\end{proof}

We will call the simplex $\tau$ as in the lemma above the \emph{projection of $\sigma$ on $Y$}. Now we are in position to define the
following notion introduced (in a more general version) by Przytycki \cite[Definition 3.1]{Pr3} in the systolic case. Let $Y$ be a convex
subset of a weakly systolic complex $\bf X$ and let
$\sigma$ be a simplex in $B_1(Y)$. The \emph{expansion by projection} $e_Y(\sigma)$
of $\sigma$ is a simplex in $B_1(Y)$ defined in the following way: if $\sigma \subseteq Y,$ then
$e_Y(\sigma)=\sigma,$ otherwise $e_Y(\sigma)$ is the join of $\sigma
\cap S_1(Y)$ and its projection on $Y$. A version of the following simple lemma was proved in \cite{Pr3} in the systolic case.
Its proof given there is valid also in our case.

\begin{lemma}[{\cite[Lemma 3.8]{Pr3}}]
\label{L3.8p3}
Let $Y$ be a convex subset of a weakly systolic complex $\bf X$ and
let $\sigma_1\subseteq \sigma_2\subseteq\ldots \subseteq
\sigma_n\subseteq B_1(Y)$ be an increasing sequence of simplices.
Then the intersection $\left( \bigcap_{i=1}^{n}e_Y(\sigma_i)\right) \cap Y$
is nonempty.
\end{lemma}

Let $\sigma$ be a simplex of a weakly systolic complex $\bf X$.
As in \cite{Pr3}, we define an increasing sequence of full subcomplexes
${\bf D}_{2i}(\sigma)$ and ${\bf D}_{2i+1}(\sigma)$ of the baricentric subdivision
${\bf X}'$ of $\bf X$ in the following way. Let ${\bf D}_{2i}(\sigma)$ be the subcomplex spanned by
all vertices of ${\bf X}'$ corresponding to simplices of $\bf X$  which have
all their vertices in $K_i(\sigma)$. Let ${\bf D}_{2i+1}(\sigma)$ be the subcomplex spanned by
all vertices of ${\bf X}'$ which correspond to those simplices of $\bf X$ that
have all their vertices in $K_{i+1}(\sigma)$ and at least one vertex in $K_i(\sigma)$.
The proof of the main proposition in this section follows closely the proof of \cite[Proposition 1.4]{Pr3}.

\begin{proposition}[Contractibility of the fixed point set]
\label{inv set contr}
Let $H$ be a group acting by simplicial automorphisms on a weakly systolic
complex $\bf X$. Then the complex $\mr {Fix}_H {\bf X}'$ is contractible or empty.
\end{proposition}

\begin{proof}
Assume that $\mr {Fix}_H{\bf X}'$ is nonempty and let $\sigma$ be a maximal
$H$--invariant simplex. By ${\bf D}_i$ we will
denote here ${\bf D}_i(\sigma)$. We will prove the following three
assertions.

\medskip\noindent
(i) ${\bf D}_0\cap \mr {Fix}_H{\bf X}'$ is contractible;

\medskip\noindent
(ii) the inclusion ${\bf D}_{2i}\cap \mr {Fix}_H{\bf X}' \subseteq {\bf D}_{2i+1}\cap
\mr {Fix}_H{\bf X}'$ is a homotopy equivalence;

\medskip\noindent
(iii) the identity on ${\bf D}_{2i+2}\cap \mr {Fix}_H{\bf X}'$ is homotopic to a
mapping with image in ${\bf D}_{2i+1}\cap \mr {Fix}_H{\bf X}'\subseteq
{\bf D}_{2i+2}\cap \mr {Fix}_H{\bf X}'$.

\medskip
As in the proof of \cite[Proposition 1.4]{Pr3}, the three assertions
imply that ${\bf D}_{k}\cap \mr {Fix}_H{\bf X}'$ is contractible for every $k$,
thus the proposition holds.  To show (i), note that ${\bf D}_{0}\cap \mr {Fix}_H{\bf X}'$
is a cone over the barycenter of $\sigma$ and hence it is contractible.

To prove (ii), let $\cal C$ be the poset of $H$--invariant simplices
in $\bf X$ with vertices in $K_{i+1}(\sigma)$ and at least one vertex in
$K_i(\sigma)$. Its geometric realization is ${\bf D}_{2i+1}\cap \mr
{Fix}_H {\bf X}'$. Consider a functor $F\colon \cal C \to \cal C$ assigning
to each object of $\cal C$ (i.e., each  simplex of $\bf X$), its subsimplex
spanned by its vertices in $K_i(A)$. By Proposition \ref{p4.1p3}, the geometric
realization of $F$ is homotopic to identity (which is the geometric
realization of the identity functor). Moreover this homotopy is
constant on ${\bf D}_{2i}\cap \mr {Fix}_H{\bf X}'$. The image of the geometric
realization of $F$ is contained in ${\bf D}_{2i}\cap \mr {Fix}_H{\bf X}'$.
Hence ${\bf D}_{2i}\cap \mr {Fix}_H{\bf X}'$ is a deformation retract of
${\bf D}_{2i+1}\cap \mr {Fix}_H{\bf X}',$ as desired.

To establish (iii), let $\cal C$ be the poset of $H$--invariant simplices
of ${\bf X}'$ with vertices in $K_{i+1}(\sigma)$ and let $\cal C'$ be its
flag poset. Let also $F_0\colon \cal C'\to \cal C$ be the functor
assigning to each object of $\cal C'$ its minimal element;
cf.\ Proposition \ref{p4.2p3}.  Now we define another functor
$F_1\colon \cal C'\to \cal C$.
For any object $c'$ of $\cal C'$, which is a chain of objects
$c_1<c_2<\ldots<c_k$ of $\cal C$, recall that $c_j$ are some
$H$--invariant simplices in $K_{i+1}(\sigma)$.
Let $c_j'=e_{K_i(\sigma)}(c_j).$ Then by Lemma \ref{L3.8p3}
the intersection $\bigcap_{j=1}^{k}c_j'$ contains at least one vertex
in $K_i(\sigma)$.
Thus $\bigcap_{j=1}^{k}c_j'$ is an $H$--invariant non-empty simplex
and hence it is an object of $\cal C$. We define $F_1(c')$ to be this
object. In the geometric realization of $\cal C$, which is
${\bf D}_{2i+2}\cap \mr {Fix}_H{\bf X}'$, the object $F_1(c')$ corresponds to a
vertex of ${\bf D}_{2i+1}\cap \mr {Fix}_H{\bf X}'$.
It is obvious that $F_1$ preserves the partial order. Notice that for any object $c'$ of $\cal C'$ we have
$F_0(c')\subseteq F_1(c')$, hence, by Proposition \ref{p4.2p3}, the geometric
realizations of $F_0$ and $F_1$ are homotopic. We have that $F_0$
is homotopic to the identity and that $F_1$ has image in ${\bf D}_{2i+1}\cap \mr
{Fix}_H{\bf X}',$ thus establishing (iii).
\end{proof}

\section{Remarks on systolic complexes}
\label{final}

In this final section, we restrict to the case of systolic complexes and present some further results in that case.
First, using Lemma 3.10 and Theorem 3.11 of Polat \cite{Po} for bridged graphs, we prove a stronger version of the fixed
point theorem for systolic complexes. Namely, Polat \cite{Po} established that for any subset $\ov Y$ of vertices  of a graph with finite intervals,
there exists a minimal isometric subgraph of this graph which contains $\ov Y.$ Moreover, if $\ov Y$ is finite and the graph is bridged, then \cite[Theorem 3.11(i)]{Po} shows that this minimal isometric (and hence bridged)  subgraph is also finite. We continue with two lemmata which can be viewed  as
$G$--invariant versions of these two results of Polat  \cite{Po}.

\begin{lemma}[Minimal subcomplex]
\label{minsubcx}
Let a group $G$ act by simplicial automorphisms on a systolic
complex $\bf X$. Let $\ov Y$ be a $G$--invariant set of vertices of $\bf X$.
Then there exists a minimal $G$--invariant subcomplex $\bf Y$ of $\bf X$
containing $\ov Y$, which is itself a systolic complex.
\end{lemma}

\begin{proof}
Let $\Sigma$ be a chain (with respect to the subcomplex relation)
of $G$--invariants subcomplexes of $\bf X$, which contain $\ov Y$ and induce isometric subgraphs of the underlying graph of $\bf X$
(and thus are systolic complexes themselves).  Then, as in the proof of \cite[Lemma 3.10]{Po}, we conclude that
the subcomplex ${\bf Y}=\bigcap \Sigma$ is a minimal $G$--invariant subcomplex
of $\bf X$, containing $\ov Y$ and which is itself a systolic complex.
\end{proof}

\begin{lemma}[Minimal finite subcomplex]
\label{minfin}

Let a group $G$ act by simplicial automorphisms on a systolic
complex $X$. Let $\ov Y$ be a finite $G$--invariant set of vertices of $X$.
Then there exists a minimal (as a simplicial complex) finite $G$--invariant subcomplex $Y$ of $X$,
which is itself a systolic complex.
\end{lemma}

\begin{proof}
Let $\mr{conv}(\ov Y)$ be the convex hull of $\ov Y$ in $\bf X$. The full subcomplex $ {\bf Z}$ of $\bf X$ spanned by $\mr{conv}(\ov Y)$ is a bounded
systolic complex. By Lemma \ref{minsubcx}, there exists a minimal $G$--invariant subcomplex $\bf Y$ of $\bf Z$ containing the set $\ov Y$ and
which itself is a systolic complex. Then,  applying  the proof of \cite[Theorem 3.11]{Po} to the bounded bridged graphs which are $1$--skeleta of the  systolic complexes $\bf Y$ of $\bf Z$, it follows that $\bf Y$ is finite.
\end{proof}

\begin{theorem}[The fixed point theorem]
\label{fpt_sc}
Let $G$ be a finite group acting by simplicial automorphisms
on a systolic complex $\bf X$. Then there exists
a simplex $\sigma \in {\bf X}$ which is invariant under the action
of $G$.
\end{theorem}

\begin{proof}
Let $\ov Y=Gv=\lk gv|\; \; g\in G\rk$, for some vertex $v\in {\bf X}$.
Then $\ov Y$ is a finite $G$--invariant set of vertices of $\bf X$ and thus, by Lemma
\ref{minfin}, there exists a minimal finite $G$--invariant subcomplex $\bf Y$ of $\bf X$,
which is itself a systolic complex. Then, the same way as in the proof of Theorem \ref{fpt},
we conclude that there exists a simplex in $\bf Y$ that is $G$--invariant.
\end{proof}

\begin{remark} We believe that,  as in the systolic case, the stronger version of Theorem \ref{fpt} holds also for weakly systolic complexes, i.e.,
one can drop the assumption on the local finiteness of $\bf X$ in Theorem \ref{fpt}. This needs extensions of some results of Polat
(in particular,  Theorems 3.8 and 3.11 from \cite{Po}) to the class of weakly bridged graphs.
\end{remark}

Zawi\' slak \cite{Z} initiated another approach to the fixed point theorem in the systolic case based on the following notion of round subcomplexes.
A systolic complex $\bf X$ of finite diameter $k$ is {\it round} (cf.\ \cite{Pr2}) if $\cap\{ B_{k-1}(v): v\in V({\bf X})\}=\emptyset.$ Przytycki \cite{Pr2} established
that all round systolic complexes have diameter at most $5$ and used this result to prove that for any finite group $G$ acting by simplicial automorphisms
on a systolic complex there exists a subcomplex   of diameter at most 5 which is invariant under the action of $G$. Zawi\' slak \cite[Conjecture 3.3.1]{Z} and Przytycki (Remark 8.1 of \cite{Pr2}) conjectured
that in fact the diameter of round systolic complexes must be at most 2. Zawi\'{s}lak \cite[Theorem 3.3.1]{Z} showed that if this is true, then it implies that
$G$ has an invariant simplex, thus paving another way to the proof of Theorem \ref{fpt_sc}. We will show now that the positive answer to the question of
Zawi\' slak and Przytycki directly follows from an earlier result of Farber \cite{Fa} on diameters and radii of finite bridged graphs.

\begin{proposition}[Round systolic complexes]\label{round}
Any round systolic complex $\bf X$ has diameter at most 2.
\end{proposition}

\begin{proof} Let diam$({\bf X})$ and rad$({\bf X})$ denote the diameter and the radius of a systolic complex $\bf X$, i.e., the diameter and radius of its underlying
bridged graph $G=G({\bf X})$. Recall that rad$({\bf X})$ is the smallest integer $r$
such that there exists a vertex $c$ of $\bf X$ (called a central vertex) so that the ball $B_r(c)$ of radius $r$ and centered at $c$ covers all
vertices of $\bf X$, i.e., $B_r(c)=V({\bf X}).$

Farber \cite[Theorem 4]{Fa} proved that if $G$ is a finite bridged graph, then $3\mr{rad}(G)\le 2\mr{diam}(G)+2.$ We will show first that this inequality holds for infinite
bridged graphs $G$ of finite diameter  $\mr{diam}(G)$ and containing no infinite simplices. Set $k:=\mr{rad}(G)\le \mr{diam}(G).$  By definition of $\mr{rad}(G)$ the
intersection of all balls of radius $k-1$ of $G$ is empty. Then using an argument of Polat (personal communication) presented below, we can find a finite subset of vertices $Y$
of $G$  such that the intersection of the balls $B_{k-1}(v),$ $v$ running over all vertices of $Y,$ is still empty.  By \cite[Theorem 3.11]{Po}, there exists a finite isometric bridged
subgraph $H$ of $G$ containing $Y.$ From the choice of $Y$ we conclude that the radius
of $H$ is at least $k$, while the diameter of $H$ is at most the diameter of $G.$
As a result, applying Farber's  inequality to $H,$ we obtain $3\mr{rad}(G)\le 3\mr{rad}(H)\le 2\mr{diam}(H)+2\le 2\mr{diam}(G)+2,$ whence
$3\mr{rad}(G)\le 2\mr{diam}(G)+2.$

To show the existence of a finite set $Y$ such that $\cap \{ B_{k-1}(v): v\in Y\}=\emptyset,$ we use an argument of Polat.
According to Theorem 3.9 of \cite{Po3}, any graph without isometric rays (in particular, any bridged graph of finite diameter) can be endowed with a topology,
called {\it geodesic topology}, so that the resulting topological space is compact. On the other hand, it is shown in \cite[Corollary 6.26]{Po4} that any convex set of a bridged
graph containing no infinite simplices is closed in the geodesic topology. As a result, the balls of a bridged graph $G$ of finite diameter containing no infinite simplices are compact
convex sets. Hence any family of balls with an empty intersection contains a finite subfamily with an empty intersection, showing that such a finite set $Y$ indeed
exists.

Now suppose that $\bf X$ is a round systolic complex and let $k:=\mr{diam}({\bf X}).$ Since $\bf X$ is round, one can easily deduce that $\mr{rad}({\bf X})=k$: indeed, if
$\mr{rad}({\bf X})\le k-1$ and $c$ is a central vertex, then $c$ will belong to the intersection $\cap\{ B_{k-1}(v): v\in V({\bf X})\},$ which is impossible.
Applying Farber's inequality to the (bridged) underlying graph of $\bf X$, we conclude that $3k\le 2k+2,$ whence $k\le 2.$
\end{proof}

\begin{remark} It would be interesting to extend Proposition \ref{round} and the relationship of \cite{Fa} between radii and diameters to weakly systolic complexes.
\end{remark}

\medskip
Osajda-Przytycki \cite{OsPr} constructed a $Z$--set compactification $\cx={\bf X} \cup \partial {\bf X}$
of a systolic complex $\bf X$. The main result there (\cite[Theorem
6.3]{OsPr}) together with Theorem E from the Introduction of our paper, suggest that for a group $G$ acting geometrically by simplicial automorphisms
on a systolic complex $\bf X$ the following result holds:
\medskip

\noindent
{\it The compactification $\cx={\bf X}\cup \partial {\bf X}$ of ${\bf X}$
satisfies the following properties:
\xms

1. $\cx$ is a Euclidean retract (ER);
\xms

2. $\partial {\bf X}$ is a $Z$--set in $\cx$;
\xms

 3. for every compact set $K\subseteq {\bf X}$,
$(gK)_{g\in G}$ is a null sequence;
\xms

 4. the action of $G$ on $\bf X$ extends to an action,
by homeomorphisms, of $G$ on $\cx$;
\xms

 5. for every finite subgroup $F$ of $G$, the fixed point set $\mr {Fix}_F \cx$ is contractible;
\xms

 6. for every finite subgroup $F$ of $G$, the fixed point set $\mr{Fix}_F {\bf X}$ is dense in $\mr {Fix}_F \cx$.
}
\medskip

This asserts that $\cx$ is an \emph{$EZ$--structure}, sensu  Rosenthal \cite{Ro}, for a systolic group $G$; for details, see \cite{OsPr}.
The existence of such a structure implies, by \cite{Ro}, the Novikov conjecture for $G$.

\section*{Acknowledgements}

Work of V. Chepoi was supported in part by the ANR grant BLAN06-1-138894 (projet
OPTICOMB). Work of D. Osajda was supported in part by MNiSW grant N N201 541738, and by the ANR grants Cannon and Th\'eorie G\'eom\'etrique des Groupes. We are grateful to Norbert Polat
for his help in the proof of Proposition \ref{round}.


\end{document}